\documentclass[a4paper,12pt]{article}
\usepackage{graphicx,amssymb,amstext,amsmath}

\usepackage[latin1]{inputenc}
\usepackage{amsthm}
\usepackage{dsfont}
\usepackage{amsmath}
\usepackage{amsfonts}
\usepackage{amssymb}
\usepackage{mathrsfs}
\usepackage{graphicx}
\usepackage{graphicx,color}
\usepackage[all]{xy}
\usepackage[total={5.2in,9.1in},
top=1.4in, left=1.55in, includefoot]{geometry}

\newtheorem{R}{Remark}

\newtheorem{T}{Theorem}

\newtheorem{C}{Corollary}

\newtheorem{Le}{Lemma}

\newcommand{\ds}{\displaystyle}
\newcommand{\B}{\mathbb{B}}
\newcommand{\A}{\mathbb{A}}
\newcommand{\re}{\mathbb{R}}
\newcommand{\und}{\underline}

\title{On the Lagrange and Markov
Dynamical Spectra.}
\author{Sergio Augusto Roma\~na Ibarra \footnote{S. A. Roma\~na is partially supported by CNPq, Capes and the Palis Balzan Prize.}\\ sergiori@im.ufrj.br\\  Carlos Gustavo T. de A. Moreira\footnote{\text{C. G. Moreira is partially supported by CNPq and the Palis Balzan Prize.}}\\ gugu@impa.br}
\date{}
\begin{document}
\maketitle

\begin{abstract}
\noindent We consider the Lagrange and the Markov dynamical spectra associated to horseshoes on surface with Hausdorff dimensions greater than $1$. We show that for a ``large" set of real functions on the surfaces and for ``typical" horseshoes with  Hausdorff dimensions greater than $1$, both the Lagrange and the Markov dynamical spectrum have persistently non-empty interior.
\end{abstract}
 \section{Introduction}
Regular Cantor sets on the line play a fundamental role in dynamical systems and notably also in some problems in number theory. They are defined by expansive maps and have some kind of self similarity property: Small parts of them are diffeomorphic to big parts with uniformly bounded distortion (see precise definition in section \ref{sec RCSEMAH}).
Some background on the regular Cantor sets with are relevant to our work can be found in \cite{CF}, \cite{PT}, \cite{MY} and \cite{MY1}.\\

\noindent A mathematical object intimately related to our work (cf. \cite{CF}), is the classical Lagrange spectrum, which is defined as  follows: Given an irrational number $\alpha$, according to Dirichlet's theorem the inequality $\left|\alpha-\frac{p}{q}\right|<\frac{1}{q^2}$ has infinite rational solutions $\frac{p}{q}$. 
Markov and Hurwitz improved this result (cf. \cite{CF}), proving that, for all irrational $\alpha$, the inequality $\left|\alpha-\frac{p}{q}\right|<\frac{1}{\sqrt{5}q^2}$ has infinitely many rational solutions $\frac{p}{q}$.\\

\noindent Meanwhile, for a fixed  irrational $\alpha$ better results can be expected. This leads us to associate, to each $\alpha$, its best constant of approximation (Lagrange value of $\alpha$), given by \begin{eqnarray*}k(\alpha)&=&\sup\left\{k>0:\left|\alpha-\frac{p}{q}\right|<\frac{1}{kq^2} \ \text{has infinitely many rational solutions $\frac{p}{q}$ }\right\}\\
&=&\limsup_{ \stackrel{|p|,q\to\infty}{p\in \mathbb{Z},q\in \mathbb N}}\left|q(q\alpha-p)\right|^{-1}\in \re\cup\{+\infty\}.
\end{eqnarray*}
\noindent Then, we always have $k(\alpha)\geq \sqrt{5}$.
Consider the set $$L=\{k(\alpha):\alpha\in \re\setminus \mathbb{Q} \ \text{and} \ k(\alpha)<\infty\}$$
known as \textit{{Lagrange spectrum}} (for properties of $L$  cf. \cite{CF}).\\

\noindent In 1947 Marshall Hall (cf.\cite{Hall}) proved that the regular Cantor set $C(4)$ of the real numbers in $[0,1]$ in whose continued fraction only appear coefficients $1, 2, 3, 4$, then 
$$C(4)+C(4)=[\sqrt{2}-1,4(\sqrt{2}-1)].$$

\noindent
Let $\alpha$ be an  irrational number expressed in continued fractions by $\alpha=[a_0,a_1,\dots]$. Define, for each $n\in \mathbb{N}$, $\alpha_{n}=[a_n,a_{n+1},\dots]$ and $\beta_{n}=[0,a_{n-1},a_{n-2},\dots]$. Using elementary continued fractions techniques it can be proved that 
$$k(\alpha)=\limsup_{n\to \infty}(\alpha_n+\beta_{n}).$$

\noindent
With this latter characterization of the Lagrange spectrum and from Hall's result it follows that $L\supset[6,+\infty)$, the Lagrange spectrum contains a whole half-line, called  \textit{Hall's ray} 
.\\

\noindent
In 1975, G. Freiman (cf. \cite{F} and \cite{CF}) proved some difficult results showing that the arithmetic sum of certain (regular) Cantor sets, related to continued fractions, contain intervals, and used them to determined the precise beginning of Hall's ray (The biggest half-line contained in $L$) which is 
$$
\frac{2221564096 + 283748\sqrt{462}}{491993569} \cong 4,52782956616\dots\,.
$$

\noindent 	Another interesting set is the classical \textit{Markov spectrum} defined by (cf. \cite{CF})
\begin{equation*}\label{SMClassic}
M=\left\{\inf_{(x,y)\in \mathbb{Z}^{2}\setminus(0,0)}|f(x,y)|^{-1}:f(x,y)=ax^2+bxy+cy^2 \ \text{with} \ b^2-4ac=1\right\}.\end{equation*}

\noindent Both the Lagrange and Markov spectrum have a dynamical interpretation.  This fact is an important motivation for our work.\\

\noindent
Let $\Sigma=({\mathbb{N}^*})^{\mathbb{Z}}$ and $\sigma\colon \Sigma \to \Sigma$ the shift defined by $\sigma((a_n)_{n\in\mathbb{Z}})=(a_{n+1})_{n\in\mathbb{Z}}$. If $f\colon \Sigma \to \re$ is defined by $f((a_n)_{n\in\mathbb{Z}})=\alpha_{0}+\beta_{0}=[a_0,a_1,\dots]+[0,a_{-1},a_{-2},\dots]$, then 

$$L=\left\{\limsup_{n\to\infty}f(\sigma^{n}(\underline{\theta})):\underline{\theta}\in \Sigma\right\}$$
and  
$$M=\left\{\sup_{n\in\mathbb{Z}}f(\sigma^{n}(\underline{\theta})):\underline{\theta}\in \Sigma\right\}.$$
\noindent This last interpretation, in terms of shift, admits a natural generalization of Lagrange and Markov spectrum in the context of hyperbolic dynamics (at least in dimension 2, which is the focus of this work).\\

\noindent
We will define the Lagrange and Markov dynamical spectrum as follows. Let $\varphi \colon M^2\to M^2$ be a diffeomorphism with $\Lambda\subset M^2$ a hyperbolic set for $\varphi$. Let $f\colon M^2\to \re$ be 
a continuous real function, then the \textit{Lagrange Dynamical Spectrum} associated to $(f,\Lambda)$ is defined by 
$$\ds L(f,\Lambda)=\left\{\limsup_{n\to\infty}f(\varphi^{n}(x)):x\in \Lambda\right\},$$

\noindent
and the \textit{Markov Dynamical Spectrum} associate to $(f,\Lambda)$ is defined by  
$$M(f,\Lambda)=\left\{ \sup_{n\in \mathbb{Z}}f(\varphi^{n}(x)):x\in \Lambda\right\}.$$

\noindent The problem of finding intervals in the classical Lagrange and Markov spectra is closely related to the study of the fractal geometry of regular Cantor sets related to the Gauss map. Fractal geometry of Cantor sets is also the key to solve some problems about dynamical Lagrange and Markov spectra in dimension two. In fact,
using results on stable intersections of two regular Cantor sets whose sum of Hausdorff dimensions is greater than 1 (cf. \cite{MY} and \cite{MY1}), we 
prove the following theorem:\\ 
\ \\
\noindent {\bf Main Theorem} \emph{Let $\Lambda$ be a horseshoe associated to a $C^2$-diffeomorphism $\varphi$ such that  $HD(\Lambda)>1$. Then there is, arbitrarily close to $\varphi$, a diffeomorphism $\varphi_{0}$ and a $C^{2}$-neighborhood $W$ of $\varphi_{0}$ such that, if $\Lambda_{\psi}$ denotes the continuation of $\Lambda$ associated to $\psi\in W$, there is an open and dense set $H_{\psi}\subset C^{1}(M,\re)$ such that for all $f\in H_{\psi}$, we have 
\begin{equation*}
int \ L(f,\Lambda_{\psi})\neq\emptyset \ \text{and} \ int  \ M(f,\Lambda_{\psi})\neq\emptyset,
\end{equation*}
where $int A$ denotes the interior of $A$.}\\

\noindent $\bf{Remark}:$ In the previous statement, by horseshoe
we mean a compact, locally maximal, transitive hyperbolic invariant set of saddle type (and so it contains a dense subset of periodic orbits).

\section{Preliminaries}

\subsection {Preliminaries from dynamical systems}\label{Prel}

\noindent If $\Lambda$ is a hyperbolic set associated to a $C^{2}$-diffeomorphism, then the stable and unstable foliations $\mathcal{F}^{s}(\Lambda)$ and $\mathcal{F}^{u}(\Lambda)$ are $C^{1}$. Moreover, these foliations can be extended to $C^{1}$ foliations defined on a full neighborhood of $\Lambda$ (cf. \cite[pp 604]{K}).\\
Let $\Lambda$ be a horseshoe of $\varphi$ and consider a finite collection $(R_{a})_{a\in\mathbb{A}}$ of disjoint rectangles of $M$, which form a Markov partition of $\Lambda$ (cf. \cite[pp 129]{Shub}). The set $\mathbb{B}\subset \mathbb{A}^{2}$ of admissible transitions consist of pairs $(a_{0},a_{1})$ such that $\varphi(R_{a_{0}})\cap R_{a_{1}}\neq \emptyset$. So, we can define the following transition matrix $B$ which induces the same transitions than $\B\subset \A^{2}$
$$b_{a_ia_j}=1 \ \ \text{if} \ \  \varphi(R_{a_i})\cap R_{a_j}\neq \emptyset,\ \ \ b_{a_ia_j}=0 \ \ \ \text{otherwise, for $(a_i,a_j)\in \A^{2}$.}$$

\noindent Let $\Sigma_{\mathbb{A}}=\left\{\und{a}=(a_{n})_{n\in \mathbb{Z}}:a_{n}\in \mathbb{A} \ \text{for all} \ n\in \mathbb{Z}\right\}$. We can define the homeomorphism of $\Sigma_{\mathbb{A}}$, the shift, $\sigma:\Sigma_{\mathbb{A}}\to\Sigma_{\mathbb{A}}$ defined by $\sigma(({a}_{n})_{n\in \mathbb{Z}})=({a}_{n+1})_{n\in \mathbb{Z}}$. \\
Let $\Sigma_{B}=\left\{\und{a}\in \Sigma_{\mathbb{A}}:b_{a_{n}a_{n+1}}=1\right\}$, this set is a closed and $\sigma$-invariant subspace of $\Sigma_{\mathbb{A}}$. Still denote by $\sigma$ the restriction of $\sigma$ to $\Sigma_{B}$. The pair $(\Sigma_{B},\sigma)$ is called a subshift of finite type of $(\Sigma_\A, \sigma)$.
\noindent Given $x,y\in \A$, we denote by $N_{n}(x,y,B)$ the number of admissible strings for $B$ of length $n+1$, beginning at $x$ and ending with $y$. Then the following holds
$$N_{n}(x,y,B)=b^{n}_{xy}.$$  
In particular, since $\varphi|_{\Lambda}$ is transitive, there is  $N_{0}\in \mathbb{N}^*$ such that for all $x,y\in \A$, $N_{N_{0}}(x,y,B)>0$.

\noindent Subshifts of finite type also have a sort of local product structure. First we define the local stable and unstable sets: (cf. \cite[chap 10]{Shub})
\begin{eqnarray*}
W_{1/3}^{s}(\und{a})&=&\left\{\und{b}\in \Sigma_{B}: \forall n\geq 0, \ d(\sigma^{n}(\und{a}),\sigma^{n}(\und{b}))\leq 1/3\right\}\\
&=&\left\{\und{b}\in \Sigma_{B}: \forall n\geq 0, \ a_n=b_n\right\},\\
W_{1/3}^{u}(\und{a})&=&\left\{\und{b}\in \Sigma_{B}: \forall n\leq 0, \ d(\sigma^{n}(\und{a}),\sigma^{n}(\und{b}))\leq 1/3\right\}\\
&=&\left\{\und{b}\in \Sigma_{B}: \forall n\leq 0, \ a_n=b_n\right\},
\end{eqnarray*}
where $d(\und{a},\und{b})=\sum_{n=-\infty}^{\infty}2^{-(2\left|n\right|+1)}\delta_{n}(\und{a},\und{b})$ and $\delta_{n}(\und{a},\und{b})$ is $0$ when $a_n=b_n$ and $1$ otherwise.\\
So, if $\und{a},\und{b}\in \Sigma_{B}$ and $d(\und{a},\und{b})<1/2$, then $a_0=b_0$ and $W_{1/3}^{u}(\und{a})\cap W_{1/3}^{u}(\und{b})$ is a unique point, denoted by the bracket $[\und{a},\und{b}]=(\cdots,b_{-n},\cdots,b_{-1},b_{0},a_{1},\cdots,a_{n},\cdots)$.  

\noindent  If $\varphi$ is a diffeomorphism of a surface ($2$-manifold), then the dynamics of $\varphi$ on $\Lambda$ is topologically conjugate to a subshift $\Sigma_{B}$ defined by $B$, namely, there is a homeomorphism $\Pi\colon \Sigma_{B} \to \Lambda$ such that, the following diagram commutes
$$\xymatrix{\Sigma_{B} \ar[r]^\sigma \ar[d]_\Pi & \Sigma_{B} \ar[d]^{\Pi \ \ \ \ \ \ds \emph{i.e.}, \ \ \ \varphi\circ\Pi=\Pi\circ \sigma.}\\
\Lambda \ar[r]^\varphi & \Lambda }$$
Moreover, $\Pi$ is a morphism of the local product structure, that is, $\Pi[\und{a},\und{b}]=[\Pi(\und{a}),\Pi(\und{b})]$, (cf. \cite[chap 10]{Shub}).


\section{The Lagrange and Markov Dynamical Spectra}

Let $\varphi:M\rightarrow M$ be a diffeomorphism of a compact 2-manifold $M$ and let $\Lambda$ be a horseshoe for $\varphi$.

\ \\
\noindent{\bf{Remark:}} We have  $L(f,\Lambda)\subset M(f,\Lambda)$ for any $f\in C^{0}(M,\mathbb{R})$. 
In fact:\\
Let $a\in L(f,\Lambda)$, then there is $x_0\in \Lambda$ such that $\ds a=\limsup_{n \to +\infty}f(\varphi^{n}(x_0))$. Since $\Lambda$ is a compact set, then there is a subsequence $(\varphi^{n_{k}}(x_0))$ of $(\varphi^{n}(x_0))$ such that $\ds\lim_{k\to +\infty}\varphi^{n_k}(x_0)=y_{0}$ and 
$$a=\limsup_{n \to +\infty}f(\varphi^{n}(x_0))=\lim_{k\to +\infty}f(\varphi^{n_k}(x_0))=f(y_0).$$

\noindent {\it{Claim:}} $f(y_{0})\geq f(\varphi^{n}(y_0))$ for all $n\in \mathbb{Z}$. Otherwise, suppose there is $n_{0}\in \mathbb{Z}$ such that $f(y_{0})<f(\varphi^{n_0}(y_{0}))$. Put $\epsilon=f(\varphi^{n_0}(y_{0}))-f(y_0)$, then, since $f$ is a continuous function, there is a neighborhood $U$ of $y_0$ such that 
$$f(y_{0})+\frac{\epsilon}{2}<f(\varphi^{n_0}(z)) \ \text{for all} \ z\in U.$$
Thus, since $\varphi^{n_k}(x_0)\to y_0$, then there is $k_{0}\in \mathbb{N}$ such that $\varphi^{n_k}(x_0)\in U$ for $k\geq k_0$, therefore, 
$$f(y_{0})+\frac{\epsilon}{2}<f(\varphi^{n_0+n_k}(x_0)) \ \text{for all} \ k\geq k_0.$$
This contradicts the definition of $a=f(y_0)$.\\
\ \\

\noindent In the next sections will be given some tools to prove the Main Theorem.
\subsection{The ``Large" Subset of $C^{1}(M,\mathbb{R})$.}\label{sec1}
In this section we construct a ``large" set of functions in $C^{1}(M,\re)$ which will be useful in the proof of Main Theorem. 
\begin {T}\label{TH1}The set $$H_{\varphi}=\left\{f\in C^{1}(M,\mathbb{R}):\#M_{f}(\Lambda)=1 \ \ and, \ for \ \ z\in M_{f}(\Lambda), \ Df_{z}(e_{z}^{s,u})\neq 0\right\}$$ is open and dense,
where $M_{f}(\Lambda)=\{z\in\Lambda:f(z)\geq f(y) \ \ \forall \ y\in \Lambda\}$ and $e_{z}^{s,u}$ are unit vectors in $E_{z}^{s,u}$ of the definition of hyperbolicity, respectively. 
\end{T}
\noindent Before proving this theorem we will present some auxiliary results.

\ \\
We say that $x$ is a boundary point of $\Lambda$ in the unstable direction, if $x$ is a boundary point of $W^{u}_{\epsilon}(x)\cap \Lambda$, \emph{i.e.}, if $x$ is an accumulation point only from one side by points in $W^{u}_{\epsilon}(x)\cap \Lambda$. If $x$ is a boundary point of $\Lambda$ in the unstable direction, then, due to the local product structure, the same holds for all points in $W^{s}(x)\cap \Lambda$. So the boundary points in the $unstable$ direction are local intersections of local $stable$ manifolds with $\Lambda$. For this reason we denote the set of boundary points in the unstable direction by $\partial_{s}\Lambda$. The boundary points in the stable direction are defined similarly. The set of these boundary points is denoted by $\partial_{u}\Lambda$.\\
\ \\ 
The following theorem is due to S. Newhouse and J. Palis (cf. \cite[pp 170]{PT}).\\

\noindent {\bf Theorem [PN]}\label{TPT}
\textit{For a horseshoe $\Lambda$ as above there is a finite number of (periodic) saddle points $p^{s}_{1},...,p^{s}_{n_{s}}$ such that
$$\Lambda\cap \left(\bigcup_{i}W^{s}(p^{s}_{i})\right)=\partial_{s}\Lambda.$$
Similarly, there is a finite number of (periodic) saddle points $p^{u}_{1},...,p^{u}_{n_{u}}$ such that 
$$\Lambda\cap \left(\bigcup_{i}W^{u}(p^{u}_{i})\right)=\partial_{u}\Lambda.$$
Moreover, both $\partial_{s}\Lambda$ and $\partial_{u}\Lambda$ are dense in $\Lambda$.}




\begin{Le}\label{L1S}
The set
$$\ds \mathscr{A^\prime}=\left\{f\in C^{2}(M,\mathbb{R}):\text{there is}\ \ z\in M_{f}(\Lambda) \ \text{with} \ Df_{z}(e_{z}^{s,u})\neq 0\right\}$$
is dense in $C^{2}(M,\mathbb{R})$, where  $e_{z}^{s,u}$ are unit vectors in $E_{z}^{s,u}$ respectively.
\end{Le}

\noindent Before proving Lemma \ref{L1S}, we remember  the definition of Morse functions.
\noindent Let $f\colon M \to \mathbb{R}$, $C^{r}, \ r\geq2$, we say that $f$ is a Morse function, if for all $x\in M$ such that $Df_{x}=0$ we have that 
$$D^{2}f(0)\colon T_{x}M\times T_{x}M \to \mathbb{R}$$
is nondegenerate, \emph{i.e.} if $D^{2}f(0)(v,w)=0$ for all $w\in T_{x}M$ implies $v=0$. Denote this set by $\mathscr{M}$. A known result says that
\textit{the set of Morse functions is open and dense in $C^{2}(M,\mathbb{R})$, $r \geq 2$.}  Note that in this case, the set $Crit(f)=\left\{x\in M: Df_{x}=0\right\}$ is a discrete set. In particular, since $\Lambda$ is a compact set, we have that $\# \left(Crit(f)\cap \Lambda\right)<\infty$. 

\begin{proof}[\bf{Proof of Lemma \ref{L1S}}]
It is enough to show simply that $\mathscr{A^\prime}$ is dense in $\mathscr{M}$ (\textit{the Morse functions}). Let $f_1\in \mathscr{M}$, then $\#Crit(f_1)<\infty$, so, since $int\Lambda=\emptyset$, we can find $f\in \mathscr{M}$, $C^{2}$-close to $f_1$ such that $M_{f}(\Lambda)\cap Crit(f)=\emptyset$. Therefore, if $z\in M_{f}(\Lambda)$, we have  $Df_{z}(e^{s}_{z})\neq 0$ or $Df_{z}(e^{u}_{z})\neq 0$.\\
If for some $z\in M_{f}(\Lambda)$, both $Df_{z}(e^{s}_{z})$ and $Df_{z}(e^{u}_{z})$ are nonzero, then $f\in \mathscr{A^{\prime}}$.
\\
\noindent In otherwise, suppose that $Df_{z}(e^{s}_{z})=0$ and $Df_{z}(e^{u}_{z})\neq 0$, then there are a $C^{2}$-neighborhood $\mathcal{V}$ of $f$ and a neighborhood $U$ of $z$, such that if $x\in U\cap \Lambda$ and $g\in \mathcal{V}$, then $Dg_{x}(e_{x}^{u})\neq 0$. Let $\mathcal{R}$ be a Markov partition of $\Lambda$, such that the element $R_z$ of $\mathcal{R}$  containing $z$ is contained in $U$. Without loss of generality, we can assume that $U$ is contained in a $C^{2}$-local chart $\phi:\tilde{U}\subset M\to V\subset \re^{2}$ with $U\subset \tilde{U}$ and $\tilde{U}\cap R'=\emptyset$ for all $R'\in \mathcal{R}\setminus \{R_z\}$. Observe that, since $Df_{z}(e^{s}_{z})=0$, then $z\in \partial_{s}\Lambda$, therefore the possible maximum points of $f$ in $\Lambda\cap R_{z}$ are on $W^{s}_{loc}(z)\cap \Lambda:=K^{s}$ (stable regular Cantor set),  which has zero Lebesgue measure. 
Consider the function $\psi^{s}:K^{s}\times \re \to \re^{2}$ defined by $$\psi^{s}(x,\alpha)=\nabla(f\circ\phi^{-1})(\phi(x))-\alpha 
\begin{pmatrix}
\ 0 &-1 \\ 1 & \ 0 \end{pmatrix}  D\phi_{x}(e^{s}_x),$$
where the above matrix is the orthogonal rotation. Since $\psi^{s}$ extends to a $C^{1}$-function, then the Lebesgue measure of $\psi^{s}(K^{s}\times \re)$ is zero. Therefore, there is $v\in \re^{2}$ with norm very small such that $v\notin \psi^{s}(K^{s}\times \re)$. Put $h(y)=f\circ\phi^{-1}(y)+\langle v,y \rangle$ for $y\in V$, thus $D(h\circ \phi)_{x}e^{s}_{x}=Dh_{\phi(x)}D\phi_{x}e^{s}_{x}\neq 0$ for all $x\in K^{s}$. Since $v$ can be taking with norm arbitrarily small, then $h\circ\phi$ is $C^{2}$-close to $f$ and since the function increases in the direction of its gradient, then the maximum points of $h$ in $\Lambda\cap R_z$ still can only appear in $K^{s}$. Thus $h$ satisfies the condition of lemma. 
   
\noindent Case $Df_{z}(e^{u}_{z})=0$ and $Df_{z}(e^{s}_{z})\neq 0$ is obtained analogously a function $C^{2}$-close to $f$ and in $\mathscr{A^{\prime}}$.\\
This concludes the proof of Lemma.
\end{proof}
\begin{Le}\label{L2S}
Let $f\in C^{1}(M,\mathbb{R})$ and $z\in M_{f}(\Lambda)$ such that $Df_{z}(e_{z}^{s,u})\neq 0$, then $z\in \partial_{s} \Lambda\cap \partial_{u}\Lambda$.
\end{Le}
\begin{proof}[\bf{Proof}]
Using local coordinates in $z$, we can assume that we are in $U\subset \mathbb{R}^{2}$ containing $0$. \\The hypothesis of the lemma implies that $Df_{z}\neq 0$, \emph{i.e.}, $f(z)$ is a regular value of $f$, then $\alpha:=f^{-1}(f(z))$ is a $C^{1}$-curve transverse to $W^{s}_{\epsilon}(z)$ and $W^{u}_{\epsilon}(z)$ in $z$, also, the gradient vector $\nabla f(z)$ is orthogonal a $\alpha$ in $z$.\\

\begin{figure}[htbp]
	\centering
		\includegraphics[width=0.3\textwidth]{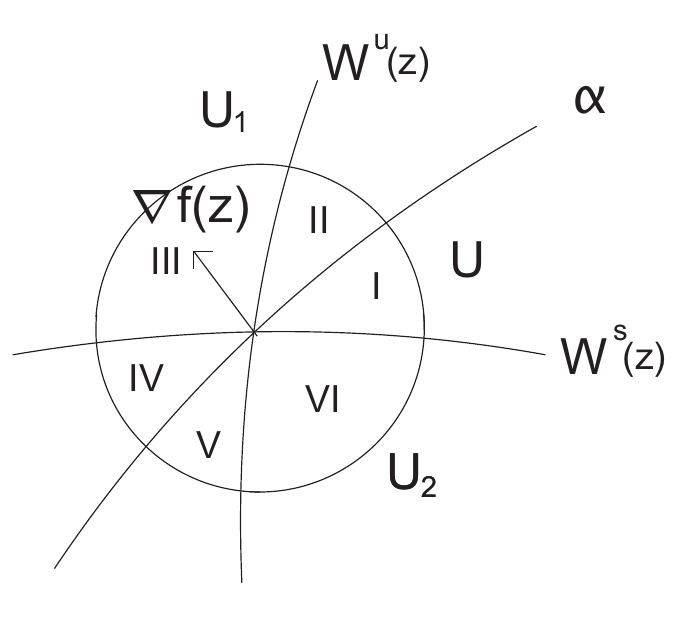}
	\caption{Localization of $z\in M_{f}(\Lambda)$}
	\label{fig:figuraspectrum2}
\end{figure}

\noindent Let $U$ be a small neighborhood $z$, then $\alpha$ subdivided into two regions $U$, say $U_{1}$, $U_{2}$ (see Figure 
\ref{fig:figuraspectrum2}). Now suppose that $\nabla f(z)$ is pointing in the direction of $U_{1}$, then in the region $I,II,III,IV$ and $V$, (see Figure \ref{fig:figuraspectrum2}), there are no points of $\Lambda$, in fact:\\
As the function increases in the direction of its gradient, then in the regions $II,III$ and $IV$, there are no points
of $\Lambda$, because $z\in M_{f}(\Lambda)$. If there are points in $I$ of $\Lambda$, then by the local product structure, there are points in $II$ of $\Lambda$, which we know can not happen. Analogously, if there are points in $V$ of $\Lambda$, then there are points in $IV$ of $\Lambda$, which we know can not happen.
In conclusion, the only region where there are points of $\Lambda$ is $VI$, so $z\in \partial_{s} \Lambda\cap \partial_{u}\Lambda$.

\end{proof}
\begin{R}\label{R1S}
\noindent Since, $C^{s}(M,\mathbb{R})$, $1\leq s\leq \infty$ is dense in $C^{r}(M,\mathbb{R})$, $0\leq r<s$, then the Lemma \ref{L1S} implies that $\mathscr{A^{\prime}}$ is dense in $C^{1}(M,\mathbb{R})$.
\end{R}

\begin{Le}\label{L4S}
The set $${H}_{1}=\left\{f\in C^{2}(M,\mathbb{R}):\#M_{f}(\Lambda)=1 \ \ and \ for \ \ z\in M_{f}(\Lambda), \ Df_{z}(e_{z}^{s,u})\neq 0\right\}$$
is dense in $C^{2}(M,\mathbb{R})$, therefore dense in $C^{1}({M,\re})$. 
\end{Le}
\begin{proof}[\bf{Proof}]
By Lemma \ref{L1S}, it is enough to show that $H_{1}$ is dense in $\mathcal{A}'$.\\
Let $f\in \mathcal{A}'$, then there is $z\in M_{f}(\Lambda)$ such that $Df_{z}(e_{z}^{s,u})\neq 0$. Take $U$ a small neighborhood of $z$. Thus, given $\epsilon>0$  small, consider the function $\varphi_{\epsilon} \in C^{2}(M,\mathbb{R})$ such that $\varphi_{\epsilon}$ is $C^{2}$-close to constant function $0$, also $\varphi_{\epsilon}=0$ in $M\setminus U$, $\varphi_{\epsilon}(z)=\epsilon$ and $z$ is a single maximum of $\varphi_{\epsilon}$. Also, $\ds \varphi_{\epsilon}\stackrel{C^{2}}{\rightarrow}0$ as $\epsilon\rightarrow 0$.\\
Define $ g_{\epsilon}=f+\varphi_{\epsilon}$, clearly $g_{\epsilon}\stackrel{C^{2}}{\rightarrow}f$ as $\epsilon\rightarrow 0$, since $z\in M_{f}(\Lambda)$ we have $g_{\epsilon}(z)=f(z)+\varphi_{\epsilon}(z)
>f(x)+\varphi_{\epsilon}(x)=g_{\epsilon}(x)$ for all $x\in \Lambda$, this is $z\in M_{g_{\epsilon}}(\Lambda)$ and $\# M_{g_{\epsilon}}(\Lambda)=1$.\\
Also, $D(g_{\epsilon})_{z}(e_{z}^{s,u})=Df_{z}(e_z^{s,u})\neq 0$, that is, $g_{\epsilon}\in H_{1}$.
\end{proof}

\begin{Le}\label{L5S}
The set $$H_{\varphi}=\left\{f\in C^{1}(M,\mathbb{R}):\#M_{f}(\Lambda)=1 \ \ and \ for \ \ z\in M_{f}(\Lambda), \ Df_{z}(e_{z}^{s,u})\neq 0\right\}$$ is open.
\end{Le}
\begin{proof}[\bf{Proof}]
Let $f\in H_{\varphi}$ and $z\in M_{f}(\Lambda)$ with $Df_{z}(e_{z}^{s,u})\neq 0$, where $e^{s,u}_{z}\in E^{s,u}_{z}$ is a unit vector, respectively. Suppose that  $\ds \frac{\partial {f}}{\partial{e_{z}^{s,u}}}=\left\langle \nabla f(z),e_{z}^{s,u}\right\rangle=Df_{z}(e_{z}^{s,u})>0$ and $\nabla f(z)$ is the gradient vector of $f$ at $z$.\\
(If we have to, $Df_{z}(e_{z}^{s})>0$ and $Df_{z}(e_{z}^{u})<0$, we consider the basis $\{e^{s}_{z},-e^{u}_{z}\}$ of $T_{z}M$ or vice versa).

\vspace{0.2cm}
\noindent Let $\mathcal{U}\subset C^{1}(M,\mathbb{R})$ an open neighborhood of $f$ such that, for all $g\in \mathcal{U}$ we have $\ds\frac{\partial {g}}{\partial{e_{z}^{s,u}}}>0$.
\ \\
The set $\left\{e_{z}^{s}, e_{z}^{u}\right\}$ is basis of $T_{z}M$. Let 
$$V=\left\{v\in T_{z}M:v=a_{v}e_{z}^{s}+b_{v}e_{z}^{u}, \ \ a_{v}, b_{v}\geq 0\right\}.$$ 
Let $v\in V\setminus \{0\}$, then $\ds\frac{\partial {g}}{\partial{v}}(z)=Dg_{z}(v)>0$, for any $g\in \mathcal{U}$. Since by Lemma \ref{L2S} we have that $z\in \partial_{s} \Lambda\cap \partial_{u}\Lambda$, this implies that, there is an open set $U$ of $z$ such that $g(z)>g(x)$, for all $g\in \mathcal{U}$ and all $x\in U \cap \Lambda \setminus \{z\}$.\\
\ \\
Let $\epsilon>0$ such that $\left|f(z)-f(x)\right|>\frac{\epsilon}{2}$ for $x\in \Lambda\setminus U$. 
Let $$\ds V_{\frac{\epsilon}{8}}(f)=\left\{g \in C^{1}(M,\mathbb{R}): \left\|f-g\right\|_{\infty}<\frac{\epsilon}{8} \ \text{and} \left\|Df-Dg\right\|_{\infty}<\frac{\epsilon}{8}\right\}$$ 
a fundamental neighborhood of $f$, then we claim that, for all $g\in V_{\frac{\epsilon}{8}}(f)$, the set $M_{g}(\Lambda)\subset U$. In fact: Let $x\in \Lambda\setminus U$, then 

\begin{eqnarray*}
g(z)-g(x) & = &g(z)-f(z)+f(z)-g(x)-f(x)+f(x) \\ 
&\geq &f(z)-f(x)-\left|g(z)-f(z)\right|-\left|g(x)-f(x)\right|\\
&\geq & \frac{\epsilon}{2}-2\frac{\epsilon}{8}=\frac{\epsilon}{4}.
\end{eqnarray*}
In particular, $g(z)>g(x), \forall x\in \Lambda \setminus U$, and so $M_g(\Lambda)=\{z\}$.

\noindent This implies that the open set $\mathcal{U}_{1}=\mathcal{U}\cap V_{\frac{\epsilon}{8}}(f)$ is contained in $H_{\varphi}$. 
\end{proof}
\noindent Now we are in condition to prove Theorem \ref{TH1}

\begin{proof}[\bf{Proof of Theorem \ref{TH1}}]
\noindent Since, $H_{1}\subset H_{\varphi}$ and by Lemma \ref{L4S} the set $H_1$ is dense in $C^{1}(M,\re)$, then $H_{\varphi}\subset C^{1}(M,\mathbb{R})$ is dense and open in $C^{1}(M,\mathbb{R})$. 
\end{proof}

\section{The Markov and Lagrange Dynamical Spectrum and Image of Sub-Horseshoes}\label{MLDSS}
In this section we prove that the Lagrange and Markov dynamical spectrum contains the image of sub-horseshoe by a real function. \\
\ \\
\noindent
Recall that the set $${H}_{\varphi}=\left\{f\in C^{1}(M,\mathbb{R}):\# M_{f}(\Lambda)=1 \ \ \text{and for} \ \ z\in M_{f}(\Lambda), \ Df_{z}({e^{s,u}_{z}})\neq 0\right\}$$
is open and dense.\\
Let $f\in H_{\varphi}$ and $x_{M}\in M_{f}(\Lambda)$, then by Lemma \ref{L2S} we have that $x_{M}\in \partial_{s}\Lambda\cap \partial_{u}\Lambda$, by Theorem [PN], we have that there are $p,q\in \Lambda$ periodic points such that 
$$x_{M}\in W^{s}(p)\cap W^{u}(q).$$

\noindent Assume that $p$ and $q$ have the symbolic representation $$(\cdots,a_{1},\cdots,a_{r},a_{1},\cdots ,a_{r},\cdots) \ \ \text{and}\ \  
(\cdots,b_{1},\cdots,b_{s},b_{1},\cdots,b_{s},\cdots)$$
respectively.\\
So, there are $l$ symbols $c_{1},\cdots,c_{l}$ such that $x_{M}$ is symbolically of the form 
$$\Pi^{-1}(x_{M})=(\cdots,b_{1},\cdots,b_{s},b_{1},\cdots,b_{s},c_{1},\cdots,c_{t},\cdots,c_{l},a_{1},\cdots,a_{r},a_{1},\cdots ,a_{r},\cdots)$$
where $c_{t}$ is the zero position of $\Pi^{-1}(x_{M})$.\\

\noindent Let $\und{q}_{\tilde{s}}=(q_{-\tilde{s}},\cdots,q_0,\cdots,q_{\tilde{s}})$ an admissible word such that $x_{M}\in R_{\und{q}_{\tilde{s}}}=\bigcap^{\tilde{s}}_{i=-\tilde{s}}\varphi^{-i}(R_{q_i})$, as in the Figure \ref{fig:figuraspectrum}, and put a sub-horseshoe $\tilde{\Lambda}:=\bigcap_{n\in \mathbb{Z}}\varphi^{n}(\Lambda\setminus R_{\und{q}_{\tilde{s}}})$, thus there exists $U$ an open set such that $U\cap\Lambda=\Lambda\setminus R_{\und{q}_{\tilde{s}}}$ and $$\tilde{\Lambda}:=\bigcap_{n\in \mathbb{Z}}\varphi^{n}(U).$$

\begin{figure}[htbp]
	\centering
		\includegraphics[width=0.50\textwidth]{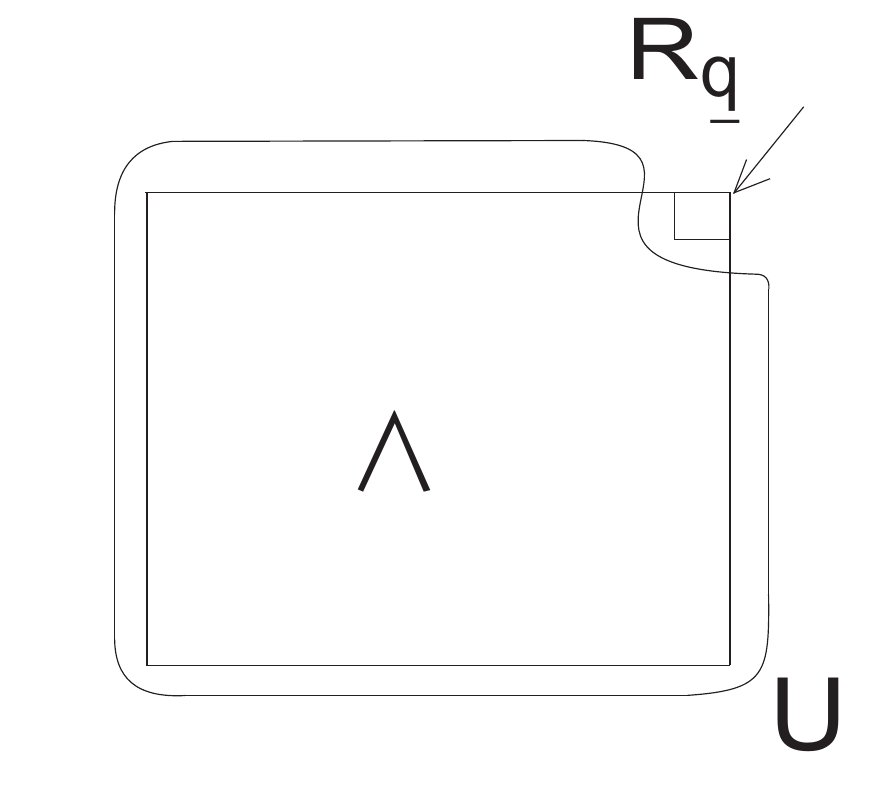}
	\caption{Removing the point of maximum}
	\label{fig:figuraspectrum}
\end{figure}

\noindent Take $\tilde{s}\in \mathbb{N}$ sufficiently large such that the Hausdorff dimension of $\tilde{\Lambda}$ is close the Hausdorff dimension of $\Lambda$, (cf. Lemma \ref{LS7}).

\noindent Let $d\in \tilde{\Lambda}$, call $\underline{d}=(\cdots,d_{-n},\cdots,d_{0},\cdots,d_{n},\cdots)$ its symbolic representation. Given $\epsilon>0$ small, take
 $n_{0}\in \mathbb{N}$ such that $\sum_{|n|\geq n_{0}}2^{-(2|n|+1)}<\epsilon$ and put  ${\underline{d}}_{n_{0}}=(d_{-n_{0}},\cdots,d_{n_{0}})$ an admissible finite word. Denote the cylinder \\ $C_{{\underline{d}}_{n_{0}}}=\left\{ \underline{w}\in \A^{\mathbb{Z}}:w_{i}=d_{i}\ \  \text{for} \ \ i=-n_{0},\cdots,n_{0}\right\}$. Then, the set
$$C_{{\underline{d}}_{n_{0}},B}:=\Sigma_{B}\cap C_{{\underline{d}}_{n_{0}}}=\left\{\underline{w}\in \Sigma_{B}:w_{i}=d_{i}\ \  \text{for} \ \ i=-n_{0},\cdots,n_{0} \right\}$$
is not empty and contains a periodic point.
\vspace{0.3cm}
\\
Since of $N_{N_{0}}(x,y,B)>0$ (cf. section \ref{Prel}), there are admissible strings $\underline{e}=(e_{1},\cdots,e_{k_0-1})$ and $\underline{f}=(f_{1},\cdots,f_{j_{0}-1})$ joining $d_{0}$ with $b_{1}$ and $a_{r}$ with $d_{1}$, respectively with $k_0,j_0<N_0$.
Since $x_{M}$ is a unique maximum point of $f$ in $\Lambda$, then if $\epsilon>0$ is small enough we can take $\tilde{\tilde{s}}>\tilde{s}$ and $\und{q}_{\tilde{\tilde{s}}}=(q_{-\tilde{\tilde{s}}},\cdots,q_0,\cdots,q_{\tilde{\tilde{s}}})$ an admissible word such that $x_{M}\in R_{\und{q}_{\tilde{\tilde{s}}}}=\bigcap^{\tilde{\tilde{s}}}_{i=-\tilde{\tilde{s}}}\varphi^{-i}(R_{q_i})\subset R_{\und{q}_{\tilde{s}}}$ and 

	\begin{equation}\label{In1}
	\sup \tilde{f}|_{\Pi^{-1}(\tilde{\Lambda})_{\epsilon}}<\inf \tilde{f}|_{\Pi^{-1}(R_{\und{q}_{\tilde{\tilde{s}}}}\cap \Lambda)},
	\end{equation}
	where $\tilde{f}=f\circ \Pi$ and $\Pi^{-1}(\tilde{\Lambda})_{\epsilon}=\{\und{x}\in \Sigma_{B}:d(\und{x}, \Pi^{-1}(\tilde{\Lambda}))<\epsilon\}$. 

\noindent Let $k\in\mathbb{N}$, $k>N_{0}$, and $k(s+r)+l>\tilde{\tilde{s}}$, then given the word $(a_{1},\cdots,a_{r})$ and $(b_{1},\cdots,b_{s})$, we define the word 
$$(a_{1},\cdots,a_{r})^{k}=\underbrace{(a_{1},\cdots,a_{r},\cdots \cdots,a_{1},\cdots,a_{r})}_{k \ times}$$
and 
$$(b_{1},\cdots,b_{s})^{k}=\underbrace{(b_{1},\cdots,b_{s},\cdots\cdots,b_{1},\cdots,b_{s})}_{k \ times}.$$
Put the word  $$\alpha=((b_{1},\cdots,b_{s})^{k},c_{1},\cdots,c_t,\cdots,c_{l},(a_{1},\cdots,a_{r})^{k}),$$
where $c_{t}$ is the zero position of the word $\alpha$.

\noindent So, fixed the words $\und{e}$ and $\und{f}$, we can define the following application, defined for all $\und{x}\in C_{{\underline{d}}_{n_{0}},B}$ by 
\begin{eqnarray*}
A(\und{x})=(\cdots,x_{-1},x_{0},e_{1},\cdots,e_{k_0-1},(b_{1},\cdots,b_{s})^{k},c_{1},\cdots,\text{\ \ \ \ \ \ \ }\\ ,c_t,
\cdots ,c_{l},(a_{1},\cdots,a_{r})^{k},f_{1},\cdots,f_{j_{0}-1},x_{1},x_{2},\cdots), 
\end{eqnarray*}
where $c_{t}$ is the zero position of the word $A(\und{x})$.
Given a finite word $\underline{a}=(a_{1},\cdots,a_{n})$, denote by $\left|\underline{a}\right|=n$, the length of the word $\underline{a}$. Then, since $k>N_{0}\geq \max\{k_{0},j_{0}\}$, we have 
$$\left|\underline{e}\right|,\left|\underline{f}\right|, \tilde{\tilde{s}}<\left|\alpha\right|=k(s+r)+l.$$
where $\underline{e}=(e_{1},\cdots, e_{k_0-1})$ and $\underline{f}=(f_{1},\cdots, f_{j_0-1})$.\\
Now we may characterize $\sup_{n\in \mathbb{Z}}\tilde{f}(\sigma^{n}(A(\und{x})))$ for $\und{x}\in C_{{\underline{d}}_{n_{0}},B}\cap \Pi^{-1}(\tilde{\Lambda})$. In fact:

\noindent Observe that $(\sigma^{l-t+kr+j_0-1}(A(\und{x})))^{+}=\und{x}^+$, call $\tau=l-t+kr+j_0-1$, then by the choice of $n_{0}$ we have that 
$d(\sigma^{\tau+n_0+n}(A(\und{x})), \sigma^{n_0+n}(\und{x}))<\epsilon$  for all $n\geq 0$. Analogously, call $\eta=-(t+sk+k_0-1)$, then $d(\sigma^{\eta-n_0-n}(A(\und{x})), \sigma^{-n_0-n}(\und{x}))<\epsilon$ for all $n\geq 0$. Moreover,  since $\Pi^{-1}(\tilde{\Lambda})$ is a $\sigma$-invariant set, then if $\und{x}\in \Pi^{-1}(\tilde{\Lambda})$ the inequality (\ref{In1}) implies that 

$$\tilde{f}(\sigma^{\tau+n_0+n}(A(\und{x}))),  \tilde{f}(\sigma^{\eta-n_0-n}(A(\und{x})))< \inf \tilde{f}|_{\Pi^{-1}(R_{\und{q}_{\tilde{\tilde{s}}}}\cap \Lambda)} \ \text{for all} \ n\geq 0.$$

\noindent The inequality above, implies that for all $\und{x}\in C_{{\underline{d}}_{n_{0}},B}\cap \Pi^{-1}(\tilde{\Lambda})$ there is \\$j\in \{\eta-n_0,\dots, \tau+n_0\}$  such that  $\sup_{n\in\mathbb{Z}}\tilde{f}(\sigma^{n}(A(\und{x})))=\tilde{f}(\sigma^{j}(A(\und{x})))$.

\noindent Put $\Pi^{-1}(x)=\und{x}$, define the set  
$$\tilde{\Lambda}_{j}:=\{x\in\tilde{\Lambda}\cap\Pi(C_{{\underline{d}}_{n_{0}},B}):  \sup_{n\in\mathbb{Z}}\tilde{f}(\sigma^{n}(A(\und{x})))=f(\sigma^{j}(A(\und{x})))\}.$$
Thus, 
\begin{equation}\label{Ajeitar}
\ds\tilde{\Lambda}\cap\Pi(C_{{\underline{d}}_{n_{0}},B})=\bigcup_{j=\eta-n_0}^{\eta+n_0}\tilde{\Lambda}_{j}.
\end{equation}
The equality (\ref{Ajeitar}) implies that there is $i_{0}\in\{\eta-n_0,\dots,\tau+n_0\}$ such that $\tilde{\Lambda}_{i_0}$ has non empty interior in $\tilde{\Lambda}\cap\Pi(C_{{\underline{d}}_{n_{0}},B})$, so
\begin{equation}\label{Ajeitar1}
HD(\ds\tilde{\Lambda})=HD(\ds\tilde{\Lambda}\cap\Pi(C_{{\underline{d}}_{n_{0}},B})=HD(\tilde{\Lambda}_{i_0}).
\end{equation}
Therefore, for $\und{x}\in \Pi^{-1}(\tilde{\Lambda}_{i_0})$ we have 
\begin{equation}\label{E1S}
\sup_{n\in \mathbb{Z}}\tilde{f}(\sigma^{n}(A(\und{x})))=\tilde{f}(\sigma^{i_0}(A(\und{x}))).
\end{equation}

\ \\ 
\noindent The next goal is to show that $\tilde{A}=\Pi \circ A\circ\Pi^{-1}$ extends to a local diffeomorphism.\\
\ \\
\noindent First we show that $\tilde{A}$ extends to a local diffeomorphism in stable and unstable manifolds of $d$, $W^{s}_{loc}(d)$ and $W^{u}_{loc}(d)$.
\ \\

\noindent As $\Lambda$ is symbolically the product $\Sigma_{\B}^{-}\times \Sigma_{\B}^{+}$, (cf. Appendix \ref{sec EMAH}), put $\beta$ the finite word ($\beta=\und{e}\alpha \und{f})$. Using the notation of Appendix \ref{sec RCSEMAH}, we have that for $$x^{u}\in W^{u}_{loc}(d)\cap \Lambda,  \ \ \text{then}  \ \    \ f^{u}_{\beta}(x^{u})\in W^{u}(d)\cap \Lambda \ \ \text{and} \ \ (\Pi^{-1}(f^{u}_{\beta}(x^{u})))^{+}=\beta(\Pi^{-1}(x^{u}))^{+},$$ also   $$x^{s}\in W^{s}_{loc}(d)\cap \Lambda, \ \ \text{then} \ \ f^{s}_{\beta}(x^{s})\in W^{s}(d)\cap\Lambda \ \  \text{and} \ \ (\Pi^{-1}(f^{s}_{\beta}(x^{s})))^{-}=(\Pi^{-1}(x^{s}))^{-}\beta.$$
The position zero of $\Pi^{-1}\left(\varphi^{-\left|\beta\right|+1}\left(f^{s}_{\beta}(x^{s})\right)\right)$, is equal to $(\beta)_{0}=e_{1}$, this is
$$\left(\Pi^{-1}\left(\varphi^{-\left|\beta\right|+1}\left(f^{s}_{\beta}(x^{s})\right)\right)\right)_{0}=(\beta)_{0}=\left(\Pi^{-1}\left(f^{u}_{\beta}(x^{u})\right)\right)_{0}.$$
So, we can define the bracket $$\ds\left[\Pi^{-1}\left(f^{u}_{\beta}(x^{u})\right),\Pi^{-1}\left(\varphi^{-\left|\beta\right|+1}\left(f^{s}_{\beta}(x^{s})\right)\right)\right]=\left(\Pi^{-1}(x^{s})\right)^{-}\beta \left(\Pi^{-1}(x^{u})\right)^{+}=A\left[\Pi^{-1}(x^{u}),\Pi^{-1}(x^{s})\right].$$
Note that for $x^{u},x^{s}$ sufficiently close to $d$ the bracket $\left[\Pi^{-1}(x^{u}),\Pi^{-1}(x^{s})\right]$ is well defined.\\
As $\Pi$ is a morphism of the local product structure, then 
\begin{eqnarray}\label{E2S}
\left[f^{u}_{\beta}(x^{u}),\varphi^{-\left|\beta\right|+1}\left(f^{s}_{\beta}(x^{s})\right)\right]&=&\Pi\left(\left[\Pi^{-1}\left(f^{u}_{\beta}(x^{u})\right),\Pi^{-1}\left(\varphi^{-\left|\beta\right|+1}\left(f^{s}_{\beta}(x^{s})\right)\right)\right]\right)\nonumber \\ 
&=& \Pi(A\left[\Pi^{-1}(x^{u}),\Pi^{-1}(x^{s})\right])=\tilde{A}\left[x^{u},x^{s}\right].
\end{eqnarray}

\noindent Put $\tilde{A}_{1}(x^{u})=f^{u}_{\beta}(x^{u})$ and $\tilde{A}_{1}(x^{s})=\varphi^{-\left|\beta\right|+1}(f^{s}_{\beta}(x^{s}))$, therefore, $\ds \tilde{A}\left[x^{u},x^{s}\right]=[\tilde{A}_{1}(x^{u}),\tilde{A}_{2}(x^{s})]$. Thus, we have the following lemma.
\begin{Le}\label{L6S}
If $\varphi$ is a $C^{2}$-diffeomorphism, then $\tilde{A}$ extends to a local $C^{1}$-diffeomorphism defined in neighborhood $U_{d}$ of $d$. We may assume without loss of generality (increasing $n_0$, if necessary) that $U_d\supset \ds\tilde{\Lambda}\cap\Pi(C_{{\underline{d}}_{n_{0}},B})$.
\end{Le}
\begin{proof}[\bf{Proof}]
As $\varphi$ is a $C^{2}$-diffeomorphism of a closed surface, then the stable and unstable foliations of the horseshoe $\Lambda$, $\mathscr{F}^{s}(\Lambda)$ and $\mathscr{F}^{u}(\Lambda)$  can be extended to $C^{1}$ invariant foliations defined on a full neighborhood of $\Lambda$. Also, if $\varphi$ is a $C^{2}$-diffeomorphism, then $f^{s}_{\beta}$ and $f^{u}_{\beta}$ are at least $C^{1}$, then by  (\ref{E2S}) we have the result.
\end{proof}
\noindent An immediate consequence of Lemma \ref{L6S} and the equality (\ref{E1S}) is:
\begin{C}\label{C5S}
If $x\in \tilde{\Lambda}_{i_0}$, then $\sup_{n\in\mathbb{Z}}f(\varphi^{n}(\tilde{A}({x})))=f(\varphi^{i_{0}}(\tilde{A}({x})))$. 
\end{C}
\noindent This Corollary implies that $\{f(\varphi^{i_{0}}(\tilde{A}({x}))):x\in \tilde{\Lambda}_{i_0}\}\subset M(f,\Lambda)$. 
\begin{R}\label{tilde A}
We have $Df_{x_{M}}(e_{x_{M}}^{s,u})\neq 0$, so this property is true in neighborhood of $x_{M}$. Since, for every $x\in \tilde{\Lambda}_{i_0}$, $\varphi^{i_{0}}(\tilde{A}({x}))$ belongs to a small neighborhood of $x_M$, then for every $x\in \tilde{\Lambda}_{i_0}$, $Df_{\varphi^{{i_0}}(\tilde{A}(x))}(e^{s,u}_{\varphi^{{i_0}}(\tilde{A}(x)})\neq 0$. Moreover, $D\varphi^{{i_0}}_{\tilde{A}(x)}(e^{s,u}_{\tilde{A(x)}})\in E^{s,u}_{\varphi^{{i_0}}(\tilde{A}(x))}$ and since by construction of $\tilde{A}$, we have that $\dfrac{\partial \tilde{A}}{\partial e_{x}^{s,u}}\parallel e^{s,u}_{\tilde{A}(x)}$, then for every $x\in \tilde{\Lambda}_{i_0}$ we have that  
$D(f\circ \varphi^{i_{0}} \circ\tilde{A})_x(e_{x}^{s,u})\ne 0$.
\end{R}
\ \\
\noindent Now we will prove the same for the Lagrange spectrum.

\vspace{0.3cm}
\noindent  Let $\epsilon$, $n_0$, and $\tilde{\tilde{s}}$ are as above. Then, using the above notation, let $x\in\tilde{\Lambda}\cap \Pi(C_{{\underline{d}}_{n_{0}},B})$ and $\Pi^{-1}(x)=(\cdots,x_{-n},\cdots,x_{0},\cdots,x_{n},\cdots)$. Thus, there is admissible string ${E}_{i}=(e_{1}^{i},\cdots, e_{s_{i}}^{i})$ joining $x_{i}$ with $x_{-i}$ the lenght $\left|{E}_{i}\right|=m_{i}-1<N_{0}$ for each $i$ (cf. section \ref{Prel}).

\noindent So, we can define the following map for all $\und{x}\in C_{{\underline{d}}_{n_{0}},B}$ by
\begin{eqnarray*}
A_{1}(\und{x})=(\cdots,x_3,E_3,x_{-3},x_{-2},x_{-1},x_{0},\beta,x_{1},x_{2},E_{2},x_{-2},x_{-1},x_{0},\beta, x_{1},E_{1},x_{-1},x_{0},\\ ,\beta,x_{1},E_{1},x_{-1},x_{0},\beta,x_{1},x_{2},E_{2},x_{-2},x_{-1},x_{0}, \beta,x_{1},x_2,x_3,E_3,x_{-3},\cdots)
\end{eqnarray*}
where $\beta=\und{e}\alpha \und{f}$ as above.\\

\noindent Since $|E_{i}|<N_0$, then the set of words $\{E_i:i\in \mathbb{N}^*\}$ is finite, therefore 
$$\{E_i:i\in \mathbb{N}^*\}=\{D_1,\dots D_{m}\},$$
for some admissible words $D_{i}$ with $|D_{i}|<N_0$.
Now we can take $k>N_0+2n_0$ and if necessary increassing $\tilde{\tilde{s}}$, we have that since $|D_{i}|<N_{0}$, then 
for each $i$, there exists a neighborhood $\mathcal{U}_i$ of $D_i$ for which
\begin{equation}\label{In2}
\sup \tilde{f}|_{\sigma^{r}(\mathcal{U}_i)}<\inf \tilde{f}|_{\Pi^{-1}(R_{\und{q}_{\tilde{\tilde{s}}}}\cap \Lambda)} \ \ \text{for} \ \ |r|\leq n_{0}+|D_i|<n_0+N_0.
\end{equation}
Now we may characterize  $\limsup_{n\to \infty}\tilde{f}(\sigma^{n}(A_1(\und{x})))$ for $\und{x}\in C_{{\underline{d}}_{n_{0}},B}\cap \Pi^{-1}(\tilde{\Lambda})$, in fact:\\
Let $m(n)\in \mathbb{N}$ such that 
$$(\sigma^{m(n)}(A_{1}(\und{x})))^{+}=x_1,x_2,\cdots,x_nE_{n}x_{-n},\cdots,x_0,\cdots \ \ \text{and} \ \ n\geq 2n_{0}.$$
Let $k^{*}$ be such that $n-k^*=n_0$, then by definition of $n_0$ we have that $$d(\sigma^{m(n)+n_0+j}(A_1(\und{x})), \sigma^{n_0+j}(\und{x}))<\epsilon \ \text{for all} \ j=0, \dots,k^*-n_{0},$$
and 
$$d(\sigma^{m(n)+n+|E_n|+n_0+j}(A_1(\und{x})), \sigma^{-k^*+j}(\und{x}))<\epsilon \ \text{for all} \ j=0, \dots,k^*-n_{0}.$$

\noindent Moreover,  since $\Pi^{-1}(\tilde{\Lambda})$ is a $\sigma$-invariant set, then, if $\und{x}\in \Pi^{-1}(\tilde{\Lambda})$ the inequality (\ref{In1}) implies that 

$$\tilde{f}(\sigma^{m(n)+n_0+j}(A_1(\und{x})))< \inf \tilde{f}|_{\Pi^{-1}(R_{\und{q}_{\tilde{\tilde{s}}}}\cap \Lambda)} \ \text{for all} \  j=0, \dots,k^*-n_{0},$$
and
$$\tilde{f}(\sigma^{m(n)+n+|E_n|+n_0+j}(A_1(\und{x})))< \inf \tilde{f}|_{\Pi^{-1}(R_{\und{q}_{\tilde{\tilde{s}}}}\cap \Lambda)} \ \text{for all} \ j=0, \dots,k^*-n_{0}.$$
Also, 
$$\sigma^{m(n)+k^*+s}(A_{1}(\und{x}))\in \sigma^{|E_{n}^{-}|+n_0-s}(\mathcal{U}_{i(n)}) \ \text{for all} \ s=0,\dots,n_0+|E_{n}^{-}|,$$
and

$$\sigma^{m(n)+n+|E_{n}^{-}|+s}(A_{1}(\und{x}))\in \sigma^{-s}(\mathcal{U}_{i(n)}) \ \text{for all} \ s=0,\dots,n_0+|E_{n}^{+}|,$$
where $E_{n}=E_{n}^{-}E_{n}^{+}$ and  $i(n)\in\{1,\dots,m\}$.
Therefore, the inequality (\ref{In2}) implies that 
$$\tilde{f}(\sigma^{m(n)+k^*+s}(A_{1}(\und{x})))<\inf \tilde{f}|_{\Pi^{-1}(R_{\und{q}_{\tilde{\tilde{s}}}}\cap \Lambda)} \ \text{for all}\ s=0,\dots,n_0+|E_{n}^{-}|,$$
and 
$$\tilde{f}(\sigma^{m(n)+n+|E_{n}^{-}|+s}(A_{1}(\und{x})))<\inf \tilde{f}|_{\Pi^{-1}(R_{\und{q}_{\tilde{\tilde{s}}}}\cap \Lambda)} \ \text{for all} \ s=0,\dots,n_0+|E_{n}^{+}|.$$

\noindent Note that, if $n_0\leq n<2n_{0}$, then $k^{*}<n_{0}$  and the last two cases applies. Therefore, the four last inequalities above implies that for all $\und{x}\in C_{{\underline{d}}_{n_{0}},B}\cap \Pi^{-1}(\tilde{\Lambda})$ there is $j\in\{\eta-n_0, \dots,\tau+n_0\}$ and a sequence $n_{k}(j)$ with 
$$\limsup_{n\to \infty}\tilde{f}(\sigma^{n}(A_1(\und{x})))=\sup_{k}\tilde{f}(\sigma^{n_{k}(j)}(A_{1}(\und{x})))\ \text{and} \ \left(\sigma^{n_{k}(j)}(A_1(\und{x})) \right)_{0}=\left( A_{1}(\und{x})\right)_{j} $$
for all $k$, where $\eta=-(t+ks+k_0-1)$ and $\tau=(l-t+kr+j_0-1)$ as above, are the length of negative and positive part of the finite word $\beta=\und{e}\alpha \und{f}$, respectively. 

\noindent Put $\Pi^{-1}(x)=\und{x}$, define the set  
$${\Lambda}'_{j}:=\{x\in\tilde{\Lambda}\cap\Pi(C_{{\underline{d}}_{n_{0}},B}):  \limsup_{n\to \infty}\tilde{f}(\sigma^{n}(A_1(\und{x})))=\ds \sup_{k}\tilde{f}(\sigma^{n_{k}(j)}(A_1(\und{x})))\}.$$

\noindent Then, 
\begin{equation*}
\ds\tilde{\Lambda}\cap\Pi(C_{{\underline{d}}_{n_{0}},B})=\bigcup_{j=\eta-n_0}^{\tau+n_0}{\Lambda}'_{j}.
\end{equation*}
Therefore, there is $j_{0}\in\{{j=\eta-n_0,\dots,\tau+n_0}\}$ such that ${\Lambda}'_{j_0}$ has non empty interior in $\tilde{\Lambda}\cap\Pi(C_{{\underline{d}}_{n_{0}},B})$, so 
\begin{equation}\label{Ajeitar2}
HD(\ds\tilde{\Lambda})=HD(\ds\tilde{\Lambda}\cap\Pi(C_{{\underline{d}}_{n_{0}},B})=HD({\Lambda}'_{j_0}).
\end{equation}
Thus, for $\und{x}\in \Pi^{-1}({\Lambda}'_{j_0})$ we have 
\begin{equation*}
\limsup_{n\to \infty}\tilde{f}(\sigma^{n}(A(\und{x})))=\ds \sup_{k}\tilde{f}(\sigma^{n_{k}({j_0})}(A_1(\und{x})))
\end{equation*}

\noindent  So, there is a subsequence $n_{k_{m}}({j_0})$ with $n_{k_{m}}({j_0})\to \infty$, as $m\to\infty$ such that 
$$\sup_{k}\tilde{f}(\sigma^{n_{k}({j_0})}(A_{1}(\und{x})))=\lim_{m\to \infty}\tilde{f}(\sigma^{n_{k_{m}}({j_0})}(A_{1}(\und{x}))).$$
By construction of $A_{1}$, it is true that 
$$\lim_{m\to \infty} \sigma^{n_{k_{m}}({j_0})}(A_{1}(\und{x}))=\sigma^{{j_0}}(A(\und{x})),$$
where $A(\und{x})$ is defined as before.\ \\
\ \\
Therefore, $$\limsup_{n \to \infty}\tilde{f}(\sigma^{n}(A_{1}(\und{x})))=\tilde{f}(\sigma^{{j_0}}(A(\und{x}))).$$\\
As an immediate consequence we have.
\begin{C}\label{C6S}
If $x\in {\Lambda}'_{j_0}$, then $$\limsup_{n\to \infty}f(\varphi^{n}(\tilde{A_{1}}(x)))=f(\varphi^{{j_0}}(\tilde{A}(x))), \ \ \text{where} \ \ \tilde{A_{1}}=\Pi\circ A_{1}\circ \Pi^{-1}.$$
\end{C}
\noindent This Corollary implies that $\{f(\varphi^{{j_0}}(\tilde{A}(x)):x\in {\Lambda}'_{{j_0}}\}\subset L(f,\Lambda)$.\\


\section{The Image of the Product of Two  Regular Cantor Sets by a Real Function and the Behavior of the Spectra.}\label{IFR}

In this section we give a condition for the image of a horseshoe by a ``typical" real function to have nonempty interior.

\subsection{Intersections of Regular Cantor Sets}\label{SIRCS}

\ \\
\noindent Assume we are given two sets of data $(\A,\B,\Sigma,g)$, $({\A}',{\B}',\Sigma',g')$ defining regular Cantor sets $K$, $K'$. See Appendix \ref{RCS} for definitions and notations.

\noindent Let $r \in (1,+\infty]$. For $a \in \A$, denote by ${\cal{P}}^{r}(a)$ the space of $C^r$-embeddings of interval $I(a)$ into $\re$, endowed with the $C^r$ topology. The affine group $Aff(\re)$ acts by composition on the left on ${\cal{P}}^r(a)$, the quotient space being denoted by $\overline{\cal{P}}^r(a)$. We also consider ${\cal{P}}(a) = \ds\bigcup_{r>1} {\cal{P}}^r(a)$ and $\overline{\cal{P}}(a) = \ds\bigcup_{r>1} \overline{\cal{P}}^r(a)$, endowed with the inductive limit topologies.

\begin{R}\label{r=1}
In \cite{MY} is considered ${\cal{P}}^{r}(a)$ for $r \in (1,+\infty]$, but all the definitions and results involving ${\cal{P}}^{r}(a)$ can be obtained considering $r\in [1,+\infty]$.
\end{R}

\noindent Let $\mathcal{A} =(\und{\theta}, A)$, where $\und{\theta} \in \Sigma^-$ and $A$ is now an {\it affine\/} embedding of $I(\theta_0)$ into $\re$. We have a canonical map
\begin{eqnarray*}
\cal{A} & \to & {\cal{P}}^r = \bigcup_{\A} {\cal{P}}^r(a)\\
(\und{\theta},A) &\mapsto & A\circ k^{\und{\theta}} \ \ (\in {\cal{P}}^r(\theta_0)).
\end{eqnarray*}

\noindent We define as in the previous the spaces $\cal{P} = \ds\bigcup_{\A}{\cal{P}}$$(a)$ and ${\cal{P}}'= \ds\bigcup_{{\A}'} {\cal{P}}(a')$.

\noindent A pair $(h,h')$, $(h \in {\cal{P}}(a), h'\in {\cal{P}} '(a'))$ is called a {\it smooth configuration\/} for $K(a)=K\cap I(a)$, $K'(a')=K'\cap I(a')$. Actually, rather than working in the product $\cal{P} \times {\cal{P}}'$, it is better to go to the quotient $Q$ by the diagonal action of the affine group $Aff(\re)$. Elements of $Q$ are called {\it smooth relative configurations\/} for $K(a)$, $K'(a')$.

\noindent We say that a smooth configuration $(h,h') \in {\cal{P}}(a)\times {\cal{P}}(a')$ is
\begin{itemize}
\item {\it linked\/} if $h(I(a)) \cap h'(I(a')) \ne \emptyset$;
\item {\it intersecting\/} if $h(K(\und{a})) \cap h'(K(\und{a}')) \ne \emptyset$, where $K(\und{a})=K\cap I(\und{a})$ and $K(\und{a}')=K\cap I(\und{a}')$;
\item {\it stably intersecting\/} if it is still intersecting when we perturb it in $\cal{P}\times\cal{P}'$, and we perturb $(g,g')$ in $\Omega_\Sigma \times \Omega_{\Sigma'}$\,.
\end{itemize}

\noindent All these definitions are invariant under the action of the affine group, and therefore make sense for smooth relative configurations.\\

\noindent As in previous, we can introduce the spaces $\cal{A}$, ${\cal{A}}'$ associated to the limit geometries of $g$,\,\,$g'$ respectively. We denote by $\cal{C}$ the quotient of $\cal{A}\times{\cal{A}}'$ by the diagonal action on the left of the affine group. An element of $\cal{C}$, represented by $(\und\theta,A) \in \cal{A}$,\,\, $(\und{\theta}', A') \in {\cal{A}}'$, is called a relative configuration of the limit geometries determined by $\und{\theta}$, $\und{\theta}'$. We have canonical maps
\begin{eqnarray*}
\cal{A}\times{\cal{A}}'&\to & \cal{P}\times{\cal{P}}'\\
\cal{C} &\to & Q
\end{eqnarray*}
which allow to define linked, intersecting, and stably intersecting configurations at the level of $\cal{A}\times{\cal{A}}'$ or $\cal{C}$.\\

\noindent We consider the following subset $V$ of $\Omega_\Sigma \times \Omega_{\Sigma'}$\,. A pair $(g, g')$ belongs to $V$ if for any $[(\und{\theta},A), (\und{\theta}',A')] \in \cal{A} \times {\cal{A}}'$ there is a translation $R_t$ (in $\re$) such that $(R_t\circ A \circ k^{\und{\theta}}, A'\circ k^{\prime\und{\theta}'})$ is a stably intersecting configuration.\\

\noindent{\bf Theorem }[cf. \cite{MY}]:

\begin{enumerate}
\it{
	\item $V$ is open in $\Omega_\Sigma\times\Omega_{\Sigma'}$, and $V \cap (\Omega_\Sigma^\infty \times \Omega_{\Sigma'}^\infty)$ is dense (for the $C^\infty$-topology) in the set $\{(g,g'), HD(K) + HD(K') > 1\}$.

\item Let $(g,g') \in V$. There exists $d^* < 1$ such that for any $(h,h') \in \cal{P}\times{\cal{P}}'$, the set 
$$
{\cal{I}}_s = \{ t\in \re, (R_{t}\circ h, h') \text{ is a stably intersecting smooth configuration for } (g,g')\}
$$
is (open and) dense in
$$
{\cal{I}}= \{t\in\re, (R_t\circ h, h') \text{ is an intersecting smooth configuration for } (g,g')\}
$$
and moreover $HD(\cal{I}-$${\cal{I}}_s) \leq d^{*}$. The same $d^{*}$ is also valid for $(\tilde g, \tilde g')$ in a neighborhood of $(g,g')$ in $\Omega_{\Sigma} \times \Omega_{\Sigma'}$.
}
\end{enumerate}
Keeping the previous notation, we have the following theorem.
\begin{T}\label{TP1S}
Let $K$, $K'$ be two regular Cantor sets defined by expanding map $g$, $g'$. Suppose that $HD(K)+HD(K')>1$ and $(g,g')\in V$. Let $f$ be a $C^{1}$-function $f \colon U \to \re$ with $K\times K'\subset U\subset {\re}^{2}$ such that, in some point of $K\times K'$ its gradient is not parallel to any of the two coordinate axis, then $$int f(K\times K')\neq \emptyset.$$ 
\end{T}

\begin{proof}[\bf{Proof}]

By hypothesis, and by continuity of $df$, we find  a pair of periodic points $p_{1}$, $p_{2}$ of $K$ and $K'$, respectively, with addresses ${\overline{\und{a}}}_{1}=\und{a_{1}a_{1}a_{1}}...$ and $\overline{\und{a}}_{2}=\und{a_{2}a_{2}a_{2}}...$, where $\und{a}_{1}$ and $\und{a}_{2}$ are finite sequences, such that $df(p_{1},p_{2})$ is not a real multiple of $dx$ nor of $dy$. There are increasing sequences of natural number $(m_k)$, $(n_k)$ such that the intervals $I_{\und{a}^{m_k}_1}$ and $I'_{\und{a}^{n_k}_2}$ defined by the finite words ${\und{a}^{m_k}_1}$ and ${\und{a}^{n_k}_2}$, satisfy 
$$\frac{|I_{\und{a}^{m_k}_1}|}{|I'_{\und{a}^{n_k}_2}|}\in (C^{-1},C) \ \text{for some} \ C>1.$$
Thus, we can assume that $\frac{|I_{\und{a}^{m_k}_1}|}{|I'_{\und{a}^{n_k}_2}|}\rightarrow \lambda \in [C^{-1},C]$ as $k\to \infty$, define $\ds\tilde{\lambda}:=-\frac{\frac{\partial f}{\partial x}(p_1,p_2)}{\frac{\partial f}{\partial y}(p_1,p_2)}\lambda$.\newline 
 As $(K,K')\in V$, then there is $t\in \re$ such that $(\tilde{\lambda}k^{{\overline{\und{a}}}_{1}}+t,k'^{{\overline{\und{a}}}_{2}})$ is a stably intersecting configuration. So, there are $\tilde{x}\in I(({\underline{a}_1})_0)$ and $\tilde{y}\in I((\underline{a}_2)_{0})$ such that $x_0=k^{{\overline{\und{a}}}_{1}}(\tilde{x})$ and $y_0=k^{{\overline{\und{a}}}_{2}}(\tilde{y})$ with $\tilde{\lambda}x_0+t=y_0$, where $({\und{a}_i})_0$  is the zero position of the finite word ${\und{a}_i}$, for $i=1, 2$. Moreover, $\tilde{x}=g^{m_k|\und{a}_1|-1}(\bar{x})$ and $\tilde{y}=(g')^{n_k|\und{a}_{2}|-1}(\bar{y})$, for some $\bar{x}\in I_{\und{a}^{m_k}_1}$ and $\bar{y}\in I'_{\und{a}^{n_k}_2}$.
\ \\
Taking $k$ large enough, we can assume that  $df(\bar{x},\bar{y})$ is not a real multiple of $dx$ nor of $dy$. In particular $\frac{\partial f}{\partial y}(\bar{x},\bar{y})\neq 0$, then by the local submersion theorem, there exists a $C^{1}$-diffeomorphism $H(x,y)=(x,g(x,y))$ defined in neighborhood of $(\bar{x},\bar{y})$ such that $f(H(x,y))=y$. Without less of generality, we can suppose that $H$ is defined in $I_{\und{a}^{m_k}_1}\times I'_{\und{a}^{n_k}_2}$.\\
Put $g_{s}(x):=g(x,s)$; if $s_0$ is such that $f(\bar{x},\bar{y})=s_0$, then $g_{s_0}(\bar{x})=\bar{y}$.
Also, observe that $s\in f((K\cap I_{\und{a}^{m_k}_1})\times (K'\cap I'_{\und{a}^{n_k}_2}))$ is equivalent to $g_{s}(K\cap I_{\und{a}^{m_k}_1})\cap  (K'\cap I'_{\und{a}^{n_k}_2})\neq \emptyset$.\\
Thus, our problem reduces to prove that $g_{s}(K\cap I_{\und{a}^{m_k}_1})$ and $K'\cap I'_{\und{a}^{n_k}_2}$ have non-empty intersection, for $s$ close to $s_0=f(\bar{x},\bar{y})$.\\
Denote by $B_k\colon I'_{\und{a}^{n_k}_2}\to [0,1]$ and $T_{k}\colon I_{\und{a}^{m_k}_1}\to [0,1]$ the orientation-preserving affine maps given by $B_{k}(x)=\frac{1}{b'_k-a'_k}(x-a'_k)=\dfrac{1}{|I'_{\und{a}^{n_k}_2}|}(x-a'_k)$ and $T_k(x)=\frac{1}{b_k-a_k}(x-a_k)=\dfrac{1}{|I_{\und{a}^{m_k}_1}|}(x-a_k)$, where $I_{\und{a}^{m_k}_1}=[a_k,b_k]$ and $I'_{\und{a}^{n_k}_2}=[a'_k,b'_k].$\\
Then, by definition of limit geometries (cf. subsection \ref{RCS}) we have that $B_{k}(K'\cap I'_{\und{a}^{n_k}_2})$ converges to $k^{{\overline{\und{a}}}_2}(K')$ and $T_{k}(K\cap I_{\und{a}^{m_k}_1})$ converges to $k^{{\overline{\und{a}}}_1}(K)$ as regular Cantor sets.\\
Also, $B_k(g_{s_0}(K\cap I_{\und{a}^{m_k}_1}))=B_k\circ g_{s_0}\circ T_k^{-1}(T_k(K\cap I_{\und{a}^{m_k}_1}))$.\\
\ \\
\noindent \textit{Claim}: The map $B_k\circ g_{s_0}\circ T_k^{-1}$ converges to $\tilde{\lambda}x+t$ in the $C^{1}$ topology.\\
In fact:\\
We call $\epsilon_{k}=b_k-a_k=|I_{\und{a}^{m_k}_1}|$ and $\epsilon'_k=b'_{k}-a'_{k}=|I'_{\und{a}^{n_k}_2}|$
\begin{eqnarray}\label{ET2}
B_{k}\circ g_{s_{0}}\circ T_{k}^{-1}(x)&=&\frac{1}{\epsilon'_k}(g_s(\epsilon_{k}x+a_k)-a'_k) \nonumber\\ 
&=&\frac{1}{\epsilon_k'}\left(  g_{s_{0}}(a_k)+g'_{s_{0}}(a_k)\epsilon_{k}x+r(\epsilon_{k}x)-a'_{k}\right) \nonumber \\
&=& B_{k}(g_{s_0}(a_k))+g'_{s_{0}}(a_k)\frac{\epsilon_k}{\epsilon'_k}x+\frac{\epsilon_k}{\epsilon'_{k}}\frac{r(\epsilon_{k}x)}{\epsilon_k}.
\end{eqnarray}
Since $g_{s_0}(\bar{x})=\bar{y}$, then $B_{k}(g_{s_0}(\bar{x}))=B_{k}(\bar{y})=B_{k}\circ (g')^{-(n_k|\und{a}_2|-1)}(\tilde{y})$ and the definition of limit geometries implies that $B_{k}(g_{s_0}(\bar{x}))$ converges to $k^{{\overline{\und{a}}}_{2}}(\tilde{y})=y_0=\tilde{\lambda}x_0+t$, also $T_{k}(\bar{x})=T_k\circ g^{-(m_k|\und{a}_1|-1)}(\tilde{x})$ converges to $k^{{\overline{\und{a}}}_{1}}(\tilde{x})=x_0$. Therefore by equation (\ref{ET2})
$$B_{k}\circ g_{s_0}(\bar{x})=B_{k}\circ g_{s_0}\circ T_{k}^{-1}(T_{k}(\bar{x}))=B_{k}(g_{s_0}(a_k))+g'_{s_{0}}(a_k)\frac{\epsilon_k}{\epsilon'_k}T_{k}(\bar{x})+\frac{\epsilon_k}{\epsilon'_{k}}\frac{r(\epsilon_{k}T_{k}(\bar{x}))}{\epsilon_k}.$$
So, if $k\to +\infty$, then the left side of the equality above converges to  $\tilde{\lambda}x_0+t$, and since $g'_{s_{0}}(a_k)\to -\frac{\frac{\partial f}{\partial x}(p_1,p_2)}{\frac{\partial f}{\partial y}(p_1,p_2)}$, $\frac{\epsilon_k}{\epsilon'_{k}}\to \lambda$, $T_{k}(\bar{x})\to x_{0}$ and $\frac{r(\epsilon_{k}T_{k}(\bar{x}))}{\epsilon_k}\to 0$, then by definition of $\tilde{\lambda}$ and  the equality above we have that 
\begin{equation}\label{ET2,2}
B_{k}\circ g_{s_0}(a_k)\to \tilde{\lambda}x_0+t-\tilde{\lambda}x_0=t.
\end{equation}
Thus, by the equalities (\ref{ET2}) and (\ref{ET2,2}) we have 
$$\lim_{k\to +\infty}B_{k}\circ g_{s_{0}}\circ T_{k}^{-1}(x)=\tilde{\lambda}x+t$$
Moreover, since $g_{s_{0}}$ is a $C^{1}$-function, then
$$\left( B_{k}\circ g_{s_{0}}\circ T_{k}^{-1}\right)'(x)=\frac{1}{\epsilon_{k}'}g'_{s_{0}}(T_{k}^{-1}(x)).\epsilon_k \to  -\frac{\frac{\partial f}{\partial x}(p_1,p_2)}{\frac{\partial f}{\partial y}(p_1,p_2)}\lambda=\tilde{\lambda}.$$
This concludes the proof the claim.\\
\ \\
Therefore, $$B_k(g_{s_0}(K\cap I_{\und{a}^{m_k}_1}))=B_k\circ g_{s_0}\circ T_k^{-1}(T_k(K\cap I_{\und{a}^{m_k}_1}))\to \tilde{\lambda}k^{{\overline{\und{a}}}_1}(K)+t,$$ and 
$$B_{k}(K'\cap I'_{\und{a}^{n_k}_2}) \to k^{{\overline{\und{a}}}_2}(K').$$
Since $(\tilde{\lambda}k^{{\overline{\und{a}}}_{1}}+t,k^{{\overline{\und{a}}}_{2}})$ is a stably intersecting configuration,  and this property is open (cf. Remark \ref{r=1}) and $g_s(\cdot)$ is $C^{1}$-close to $g_{s_0}(\cdot)$ for $s$ close to $s_0$, then for $k$ large enough we have that the Cantor set 
$B_k(g_{s}(K\cap I_{\und{a}^{m_k}_1}))$ and $B_{k}(K'\cap I'_{\und{a}^{n_k}_2})$  have non-empty intersection, therefore 
$g_{s}(K\cap I_{\und{a}^{m_k}_1})$ and $K'\cap I'_{\und{a}^{n_k}_2}$ have non-empty intersection.
\end{proof}
\ \\
\noindent The following example shows that the property $V$ in the Theorem \ref{TP1S} is fundamental.\\
\ \\
\noindent {\bf Example:}
Consider the regular Cantor set $K_{\alpha}:=\bigcap_{n\geq 0}\psi^{-n}(I_1\cup I_2)$, where 

\begin{equation*}
 \psi(x) = \left\{ \begin{array}{lll}
         \ \ \ \frac{2}{1-\alpha}x & \mbox{\text{if} \ $x\in I_1:=[0,\frac{1-\alpha}{2}]$};\\
           & \\
       -\frac{2}{1-\alpha}x + \frac{2}{1-\alpha}& \mbox{\text{if} \ $x\in I_2:=[\frac{1+\alpha}{2},1]$}.
\end{array} \right.
\end{equation*}

\noindent Since, $HD(K_{\alpha})=-\frac{\log2}{\log(\frac{1-\alpha}{2})}$ (cf. \cite{PT}). Then, if $\alpha<1/2$, then $HD(K_{\alpha})>1/2$ and for $1/3<\alpha<1/2$ hold that $K_{\alpha}-K_{\alpha}$ has measure zero (cf. \cite{CCDA}).\\
\ \\
Moreover, $HD(K_{\alpha}\times K_{\alpha})>1$ and $f(x,y)=x-y$, satisfies the hypothesis of previous Theorem, but for $1/3<\alpha<1/2$ we have that $$int\ f(K_{\alpha}\times K_{\alpha})=\emptyset.$$

\begin{C}\label{C7S}
Let $\varphi$ be a $C^{2}$-diffeomorphism, and $\Lambda$ a horseshoe associated to $\varphi$, suppose that $K^{s}$, $K^{u}$ satisfy the hypotheses of the theorem above, put 
$$\mathcal{A}_{\Lambda}=\{f\in C^{1}(M,\mathbb{R}):\exists z=(z^{s},z^{u})\in \Lambda \ \ \text{such that} \ \  Df(z).e_{z}^{s,u} \ne 0 \ \}.$$
Then, for all $f\in \mathcal{A}_{\Lambda}$, we have $int f(\Lambda)\neq \emptyset$.
\end{C}
\noindent It is easy to prove that $\mathcal{A}_{\Lambda}$ given the above Corollary is an open and dense set in $C^{1}(M,\re)$.

\section{The Main Theorem}\label{MT}
\noindent A fundamental result due to Moreira-Yoccoz in \cite{MY1} on the existence of elements in $V$ associated to the pairs of regular Cantor sets $(K^{s},K^{u})$ defined by $g^{s}, \ g^{u}$ where $g^{s}$ describes the geometry transverse of the unstable foliation $W^{u}(\Lambda,R)$ and $g^{u}$ describes the geometry transverse of the stable foliation $W^{s}(\Lambda,R)$ given in the section \ref{sec EMAH} is the following:\\

\noindent {\bf Theorem}[cf. \cite{MY1}\label{MY1S}]
\textit{Suppose that the sum of the Hausdorff dimensions of the regular Cantor set $K^{s}, K^{u}$ defined by $g^{s},g^{u}$ is $> 1$. If the neighborhood $\mathcal{U}$ of $\varphi_{0}$ in $Diff^{\infty}(M)$ is sufficiently small, there is an open and dense $\cal{U}^{*}\subset \cal{U}$ such that for $\varphi \in \cal{U}^{*}$
the corresponding pair of expanding applications $(g, g^{\prime})$ belongs to $V$.}\\

\noindent We use the above results to show that the Markov and Lagrange spectra have typically non-empty interior in this context.\\ 

\noindent Remember that given a horseshoe $\Lambda$ associated to a diffeomorphism $\varphi$, then for $f\in H_{\varphi}$, we defined the subhorseshoe $\tilde{\Lambda}$ in section \ref{MLDSS} as $\tilde{\Lambda}:=\bigcap_{n\in \mathbb{Z}}\varphi^{n}(\Lambda\setminus R_{\und{q}_{\tilde{s}}})$.\\
The following lemma shows that the $HD(\tilde{\Lambda})$ does not change much compared to $HD(\Lambda)$. More precisely.

\begin{Le}\label{LS7}
If $\Lambda$ is a horseshoe associated to a $C^{2}$-diffeomorphism $\varphi$ and $HD(\Lambda)>1$, then $\ds HD(\tilde{\Lambda})>1$ provided $\tilde s$ is large enough.
\end{Le}
\ \\
Assuming the Lemma \ref{LS7}, then the equations (\ref{Ajeitar1}) and (\ref{Ajeitar2})  implies that:
\begin{C}\label{CLS7}
The sets $\tilde{\Lambda}_{i_0}$ and $\Lambda'_{j_0}$ satisfies $HD(\tilde{\Lambda}_{i_0}), \ HD(\Lambda'_{j_0})>1$.
\end{C}

\noindent Recalling that, as $\varphi$ is a $C^{2}$-diffeomorphism, $\Lambda$ is locally the product of stable and unstable regular Cantor set, $K^{s}\times K^{u}$. Then the previous lemma will be a consequence of the following lemma.
\ \\ 
\ \\
Let $K$ be a regular Cantor set with expanding map $\psi$ and Markov partition $\mathcal{R}=\{I_{1},\cdots,I_k\}$, so that $\ds K=\bigcap_{n\geq0}\psi^{-n}(\bigcup_{i=1}^{k}I_{i})$. Consider the transition matrix $A=(a_{ij})_{k\times k}$ associated to the partition $\mathcal{R}$, define by 

\[ a_{ij} = \left\{ \begin{array}{ll}
         1 & \mbox{if $\psi(I_{i})\supset I_{j}$};\\
				 0 & \mbox{if $\psi(I_{i})\cap I_{j}=\emptyset$} \end{array} \right.. \] 

\noindent To each admissible finite word of length $m$, $\underline{b}=(b_1,\cdots,b_m)$, such that $a_{b_{i}b_{i+1}}=1, \forall i<m$, we associate the interval $I_{\underline{b}}=I_{b_1}\cap \psi^{-1}(I_{b_{2}})\cap  \psi^{-2}(I_{b_{3}}) \cdots \cap \psi^{-(m-1)}(I_{b_{m}})$.
\begin{Le}\label {L8S}
Let $K$ be a regular Cantor set with expanding map $\psi$ and Markov partition $\mathcal{R}=\{I_{1},\cdots,I_k\}$, so that $\ds K=\bigcap_{n\geq0}\psi^{-n}(\bigcup_{i=1}^{k}I_{i})$. Given $\epsilon >0$ there is a positive integer $m_0$ such that, for every $m\ge m_0$, and for every admissible finite word of length $m$, $\underline{b}=(b_1,\cdots,b_m)$,
 $$HD({K}_{\underline{b}})\geq HD(K)-\epsilon,
\ \ \text{where}
 \ \ {K}_{\underline{b}}=\bigcap_{n\geq0}\psi^{-n}(\bigcup_{{i=1} }^{k}I_{i}\setminus I_{\underline{b}}).$$
\end{Le}
\begin{proof}[\bf{Proof}]
Let ${\mathcal{R}}^{n}$ denote the set of connected components of $\psi^{-(n-1)}(I_{i})$, $I_{i}\in {\mathcal{R}}$. Let $B^n$ be the set of admissible words of length $n$, so that ${\mathcal{R}}^{n}=\{I_{\underline{b}}, \underline{b}\in B^n\}$. Fix $\tilde{i}, \tilde{j} \le k$ such that $a_{\tilde{j}\tilde{i}}=1$. Let $X^n=\{\underline{b}=(b_1,\cdots,b_n)\in B^n: b_1=\tilde{i}, b_n=\tilde{j}\}$. For any positive integer $r$, and $\underline{b_1}, \underline{b_2},\dots,\underline{b_r} \in X^n$, we have $\underline{b_1} \ \underline{b_2}\dots\underline{b_r} \in X^{nr}\subset B^{nr}$. Let ${\mathcal{\tilde{R}}}^{n}=\{I_{\underline{b}}, \underline{b}\in X^n\}$.

For $R\in {\mathcal{R}}^{n}$ take $\Lambda_{n,R}=\ds\sup \left|(\psi^{n})^{\prime}_{|_{R}}\right|$. By the mixing condition, there is $c_1>0$ such that $$\sum_{R\in\mathcal{\tilde{R}}^{n}}(\Lambda_{n,R})^{-d}\ge c_1\sum_{R\in\mathcal{R}^{n}}(\Lambda_{n,R})^{-d}, \forall d\ge 0, n\ge 1.$$
On the other hand, from \cite[pg. 69-70]{PT}, it follows that, if we define $d_n$ implicitely by
$$\sum_{R\in\mathcal{R}^{n}}(\Lambda_{n,R})^{-d_n}=1,$$
then $\lim d_n=HD(K)$, so in particular for $n$ large we have $d_n>HD(K)/2$. Notice also that there is $\lambda_1>1$ such that $\Lambda_{n,R}\ge \lambda_1^n, \forall n\ge 1$.

Let $n$ be large so that $d_n>HD(K)-\varepsilon/2$ and $\lambda_1^{n\varepsilon/2}>2/c_1$, and let $m_0=2n-1$. Given $m\ge m_0$ and an admissible finite word of length $m$, $\underline{b}=(b_1,\cdots,b_m)$, define words $\underline{c_j}=(b_j,b_{j+1},\dots,b_{j+n-1})\in B^n, 1\le j\le n$. Let $L^n=\{\underline{c_j}: 1\le j\le n\}$ and ${\mathcal{\hat{R}}}^{n}=\{I_{\underline{c}}: \underline{c}\in X^n\setminus L_n\}$. We have
\begin{eqnarray*}
\sum_{R\in\mathcal{\hat{R}}^{n}}(\Lambda_{n,R})^{-d_n}&\ge& \sum_{R\in\mathcal{\tilde{R}}^{n}}(\Lambda_{n,R})^{-d_n}-n\lambda_1^{-n d_n}\ge \sum_{R\in\mathcal{\tilde{R}}^{n}}(\Lambda_{n,R})^{-d_n}-n\lambda_1^{-n HD(K)/2}\\
&\ge& c_1\sum_{R\in\mathcal{R}^{n}}(\Lambda_{n,R})^{-d_n}-n\lambda_1^{-n HD(K)/2}=c_1-n\lambda_1^{-n HD(K)/2}>c_1/2,
\end{eqnarray*}
and so $$\sum_{R\in\mathcal{\hat{R}}^{n}}(\Lambda_{n,R})^{-(HD(K)-\varepsilon)}>\sum_{R\in\mathcal{\hat{R}}^{n}}(\Lambda_{n,R})^{-(d_n-\varepsilon/2)}>1.$$

\noindent We may define the regular Cantor set (with expanding map $\psi^n$) $$\ds \tilde{K}:=\bigcap_{r\ge 0} \psi^{-nr}(\cup_{\hat I \in \mathcal{\hat{R}}^{n}}\hat I).$$  
The previous estimate implies that $$\sum_{R\in{\mathcal{\overline{R}}}^{nr}}(\Lambda_{nr,R})^{-(HD(K)-\varepsilon)}\ge \left(\sum_{R\in\mathcal{\hat{R}}^{n}}(\Lambda_{n,R})^{-(HD(K)-\varepsilon)}\right)^r>1,$$
where $\mathcal{\overline{R}}^{nr}=\{I_{\underline{c_1} \underline{c_2}\dots\underline{c_r}}, \underline{c_j}\in X^n\setminus L_n, \forall j\le r\}$.
An thus we conclude (as before) that $HD(\tilde K)\geq HD(K)-\epsilon$. 

For any positive integer $r$ and $\underline{b_1}, \underline{b_2},\dots,\underline{b_r} \in X^n\setminus L_n$, the sequence $(a_1, a_2,\dots, a_{nr})=\underline{b_1} \underline{b_2}\dots\underline{b_r}$ satisfies that $\forall j, 1\le j\le nr-m+1, (a_j, a_{j+1}\cdots,a_{j+m-1})\ne \underline{b}$, and so $\tilde K \subset {K}_{\underline{b}}$. In particular, $HD({K}_{\underline{b}})\geq HD(\tilde K)\geq HD(K)-\epsilon$. 
\end{proof}
\vspace{0.2cm}
\begin{proof}[\bf{Proof of Lemma \ref{LS7}}]
Apply the previous Lemma to $K^{s}$ and $K^{u}$ and then use the fact that the Hausdorff dimension of the cartesian product of regular Cantor sets is the sum of the Hausdorff dimensions of the Cantor sets. 
\end{proof}
\vspace{0.2cm}
\noindent Note that by Lemma \ref{LS7} and the local structure of $\tilde{\Lambda}$, we have $HD(\tilde{\Lambda}\cap U_d)=HD(\tilde{\Lambda})>1$, where $U_{d}$ is the small neighborhood of $d$ given by Lemma \ref{L6S}.
\subsection{Proof of the Main Theorem}\label{PMT}
\noindent Given a pair $(\varphi,\Lambda)$ of a diffeomorphism $\varphi$ and a horseshoe associate to $\varphi$ with $HD(\Lambda)>1$, we defined in the section {\ref{sec1}} (cf. Lemma \ref{L5S} and Theorem \ref{TH1}) the open dense set $H_{\varphi}$ in $C^{1}(M,\mathbb{R})$  by
$$H_{\varphi}=\left\{f\in C^{1}(M,\mathbb{R}):\#M_{f}(\Lambda)=1 \ \ \text{and for} \ \ z\in M_{f}(\Lambda), \ Df_{z}(e_{z}^{s,u})\neq 0\right\}.$$

\noindent Remember that $\tilde{\Lambda}$ is a sub-horseshoe of $\Lambda$ as in Lemma \ref{LS7} with $HD(\tilde{\Lambda}\cap U_{d})=HD(\tilde{\Lambda})>1$, then by the theorem from [MY10] which we discussed above, there is a diffeomorphism $\varphi_{0}$ close to $\varphi$, a horseshoe $\Lambda_{0}$ associated to $\varphi_{0}$ and a sub-horseshoe $\tilde{\Lambda}_{0}\subset \Lambda_{0}$ with $HD(\tilde{\Lambda}_{0})>1$ and such that $\tilde{\Lambda}_{0}$ satisfies the hypotheses of Theorem \ref{TP1S}  (we use the theorem to perturb the sub-horseshoe).
\vspace{0.2cm}
\ \\
Given $f\in H_{\varphi_{0}}$, we can define $\tilde{A}_{\varphi_{0}}(f)$ a local diffeomorphism; in coordinates given by the stable and unstable foliation, we can write, $\tilde{A}_{\varphi_{0}}(f)(x,y)=(\tilde{A}^{1}_{\varphi_{0}}(f)(x),\tilde{A}^{2}_{\varphi_{0}}(f)(y))$ (see equality (\ref{E2S})) as in the section {\ref{MLDSS}}.
\vspace{0.2cm}
\ \\
Let $i_0$ be such that the Corollary \ref{C5S} holds for $(\varphi_{0},\tilde{\Lambda}_0)$ and we have (\ref{E1S}) and (\ref{Ajeitar1}).
\noindent  For $f\in H_{\varphi_0}$, Remark 
\ref{tilde A} implies that, for every $x\in \tilde{\Lambda}_{0,i_0}$, $D(f\circ \varphi^{i_{0}}_{0} \circ\tilde{A}_{\varphi_{0}}(f))_x(\tilde{e}_{x}^{s,u})\ne 0$, where $\tilde{e}_{x}^{s,u}$ are the unit vectors in stable and unstable bundle of hyperbolic set $\tilde{\Lambda}_{0}$, respectively (here $\tilde{\Lambda}_{0,i_0}$ is defined as in (\ref{Ajeitar1}) but for $\tilde{\Lambda}_{0}$ instead of $\tilde{\Lambda}$). So, the function $f\circ \varphi^{i_{0}}_{0} \circ\tilde{A}_{\varphi_{0}}(f)\in\mathcal{A}_{\tilde{\Lambda}_{0}}$. Therefore, by Corollary \ref{C7S} 

\begin{equation}\label{E4S}
int (f\circ \varphi^{i_{0}}_{0} \circ\tilde{A}_{\varphi_{0}}(f))(\tilde{\Lambda}_{0})\neq \emptyset.
\end{equation}
\noindent Then, as in Corollary \ref{C5S}, we have that
$$\sup_{n\in \mathbb{Z}}f(\varphi_{0}^{n}(\tilde{A}_{\varphi_{0}}(f)(x))=(f\circ \varphi^{i_{0}}_{0} \circ\tilde{A}_{\varphi_{0}}(f))(x)$$
for all $x\in\tilde{\Lambda}_{0,i_0}$.
This implies that $(f\circ \varphi^{i_{0}}_{0} \circ\tilde{A}_{\varphi_{0}}(f))(\tilde{\Lambda}_{0,i_0})\subset M(f,\Lambda_{0})$.
Thus, by (\ref{E4S}) we have that $int M(f,\Lambda_{0})\neq \emptyset$.\\
\ \\
Using Corollary \ref{C6S} instead of Corollary \ref{C5S}, we get the analogous result for the Lagrange spectrum. This concludes of proof of Main Theorem.

\begin{flushright}
$\square$
\end{flushright}

\section{Appendix}\label{sec RCSEMAH}

\subsection{Regular Cantor Sets and Limit Geometries}\label{RCS}
Let $\mathbb{A}$ be a finite alphabet, $\mathbb{B}$ a subset of $\mathbb{A}^{2}$, and $\Sigma_{\mathbb{B}}$ the subshift of finite type of $\mathbb{A}^{\mathbb{Z}}$ with allowed transitions $\mathbb{B}$. We will always assume that $\Sigma_{\mathbb{B}}$ is topologically mixing, and that every letter in $\A$ occurs in $\Sigma_{\mathbb{B}}$.

\noindent An {\it expansive map of type\/} $\Sigma_{\mathbb{B}}$ is a map $g$ with the following properties:
\begin{itemize}
\item[(i)] the domain of $g$ is a disjoint union
$\ds\bigcup_{\mathbb{B}}I(a,b)$. Where for each $(a,b)$,\,\, $I(a,b)$ is a compact subinterval of $I(a) := [0,1]\times\{a\}$;
\item[(ii)] for each $(a,b) \in \mathbb{B}$, the restriction of $g$ to $I(a,b)$ is a smooth diffeomorphism onto $I(b)$
satisfying $|Dg(t)| > 1$ for all $t$.
\end{itemize} 
\noindent The {\it regular Cantor set\/} associated to $g$ is the maximal invariant set
$$K = \bigcap_{n\ge0} g^{-n}\bigg(\bigcup_{\B} I(a,b)\bigg).$$

\noindent Let $\Sigma^+_{\mathbb{B}}$ be the unilateral subshift associated to $\Sigma_{\mathbb{B}}$. There exists a unique homeomorphism $h\colon \Sigma^{+}_{\B} \to K$ such that
$$
h(\underline{a}) \in I(a_0), \text{ for } \underline{a} = (a_0,a_1,\dots) \in \Sigma^+_{\mathbb{B}}
\ \ and \ \ 
h\circ\sigma =g \circ h,$$
where, $\sigma^{+}\colon \Sigma_{\B}^{+} \to \Sigma_{\B}^{+}$, is defined as follows $\sigma^{+}((a_{n})_{n\geq 0})=(a_{n+1})_{n\geq0}$.
For $(a,b)\in\B$, let $$f_{a,b}={\left[g|_{I(a,b)}\right]}^{-1}.$$
This is a contracting diffeomorphism from $I(b)$ onto $I(a,b)$. If $\underline{a}=(a_{0},\cdots,a_{n})$ is a word of $\Sigma_{\B}$, we put 
$$f_{\underline{a}}=f_{a_{0},a_{1}}\circ \cdots \circ f_{a_{n-1},a_{n}}$$
this is a diffeomorphism from $I(a_{n})$ onto a subinterval of $I(a_{0})$ that we denote by $I(\underline{a})$, with the property that if $z$ in the domain of $f_{\underline{a}}$ we have that 
$$f_{\underline{a}}(z)=h(\underline{a}h^{-1}(z)).$$
\

\noindent Let $r>1$ be a real number, or $r=+\infty$. The space of $C^r$ expansive maps of type $\Sigma$, endowed with the $C^r$ topology, will be denoted by $\Omega_\Sigma^r$\,. The union $\Omega_\Sigma = \ds\bigcup_{r>1} \Omega_\Sigma^r$ is endowed with the inductive limit topology.\\

\noindent  Let $\Sigma^- = \{(\theta_n)_{n\leq 0}\,, (\theta_i,\theta_{i+1})
\in \B \text{ for } i < 0\}$. We equip $\Sigma^-$ with the following
ultrametric distance: for $\und{\theta} \ne \und{\widetilde\theta} \in \Sigma^-$, set

\[ d(\und{\theta},\und{\widetilde\theta}) = \left\{ \begin{array}{lll}
         \ \ \ 1 & \mbox{if \ $\theta_0 \ne \widetilde{\theta}_0$};\\
           & \\
       |I(\und{\theta} \wedge \und{\widetilde\theta})|& \mbox{otherwise}\end{array} \right.. \]

\noindent where $\und{\theta} \wedge \und{\widetilde\theta} = (\theta_{-n},\dots,\theta_0)$ if
$\widetilde\theta_{-j} = \theta_{-j}$ for $0 \le j \le n$ and $\widetilde\theta_{-n-1}
\ne \theta_{-n-1}$\,.
\\
\noindent Now, let $\und{\theta} \in \Sigma^-$; for $n > 0$, let $\und{\theta}^n = (\theta_{-n},\dots,\theta_0)$, and let $B(\und{\theta}^n)$ be the affine map from
$I(\und{\theta}^n)$ onto $I(\theta_0)$ such that the diffeomorphism $k_n^{\und{\theta}}
= B(\und{\theta}^n) \circ f_{\und{\theta}^n}$ is orientation preserving.\\

\noindent We have the following well-known result (cf. \cite{Su}):
\\

\noindent{\bf Proposition}. \textit{Let} $r \in (1,+\infty)$, $g \in \Omega_\Sigma^r$.
\begin{enumerate}
\it{
	\item For any $\und{\theta} \in \Sigma^-$, there is a diffeomorphism $k^{\und{\theta}} \in \text{Diff}_+^{\ r}(I(\theta_0))$ such that $k_n^{\und{\theta}}$ converge to $k^{\und{\theta}}$ in $\text{Diff}_{+}^{\ r'}(I(\theta_0))$, for any $r'< r$, uniformly in $\und{\theta}$. The convergence is also uniform in a neighborhood of $g$ in $\Omega_\Sigma^r$\,.

\item  If $r$ is an integer, or $r = +\infty$,\,\,\,$k_n^{\und{\theta}}$ converge to $k^{\und{\theta}}$ in $\text{Diff}_+^r(I(\theta_0))$. More precisely, for every $0 \leq j \leq r-1$, there is a constant $C_j$ (independent on $\und{\theta}$) such that
$$
\left| D^j \, \log \, D \left[k_n^{\und{\theta}} \circ (k^{\und{\theta}})^{-1}\right](x)\right| \leq C_j|I(\und{\theta}^n)|.
$$
It follows that $\und{\theta} \to k^{\und{\theta}}$ is Lipschitz in the following sense: for $\theta_0 = \widetilde\theta_0$, we have
$$
\left|D^j \, \log \, D\big[k^{\und{\widetilde\theta}} \circ (k^{\und{\theta}})^{-1}\big](x)\right| \leq C_j\,d(\und{\theta}, \und{\widetilde\theta}).
$$
}
\end{enumerate}

\subsection {Expanding Maps Associated to a Horseshoe}\label{sec EMAH}
\noindent Let $\Lambda$ be a horseshoe associate a $C^{2}$-diffeomorphism $\varphi$ on the a surface $M$ and consider a finite collection $(R_{a})_{a\in\mathbb{A}}$ of disjoint rectangles of $M$, which are a Markov partition of $\Lambda$. Put the sets 
 $$W^{s}(\Lambda,R)=\bigcap_{n\geq0}\varphi^{-n}(\bigcup_{a\in \mathbb{A}}R_{a}),$$
$$W^{u}(\Lambda,R)=\bigcap_{n\leq0}\varphi^{-n}(\bigcup_{a\in \mathbb{A}}R_{a}).$$
There is a $r>1$ and a collection of $C^{r}$-submersions $(\pi_{a}:R_{a}\rightarrow I(a))_{a\in\mathbb{A}}$, satisfying the following properties:\\
\ \\
If $z,z^{\prime}\in R_{a_{0}}\cap \varphi^{-1}(R_{a_{1}})$ and $\pi_{a_{0}}(z)=\pi_{a_{0}}(z^{\prime})$, then we have $$\pi_{a_{1}}(\varphi(z))=\pi_{a_{1}}(\varphi(z^{\prime})).$$

\noindent In particular, the connected components of $W^{s}(\Lambda,R)\cap R_{a}$ are the level lines of $\pi_{a}$. Then we define a mapping $g^{u}$ of class $C^{r}$ (expansive of type $\Sigma_{\mathbb{B}}$) by the formula
$$g^{u}(\pi_{a_{0}}(z))=\pi_{a_{1}}(\varphi(z))$$
\noindent for $(a_{0},a_{1})\in \B$, $z\in R_{a_{0}}\cap\varphi^{-1}(R_{a_{1}})$.
The regular Cantor set $K^{u}$ defined by $g^{u}$, describes the geometry transverse of the stable foliation $W^{s}(\Lambda,R)$.
Moreover, there exists a unique homeomorphism $h^{u}\colon \Sigma^{+}_{\B} \to K^u$ such that
$$
h^{u}(\underline{a}) \in I(a_0), \text{ for } \underline{a} = (a_0,a_1,\dots) \in \Sigma^+_{\B}
\ \ and \ \ 
h^{u}\circ\sigma^{+} =g^{u} \circ h^{u},$$

\noindent where $\sigma^{+}\colon \Sigma_{\B}^{+} \to \Sigma_{\B}^{+}$, is defined as follows $\sigma^{+}((a_{n})_{n\geq 0})=(a_{n+1})_{n\geq0}$.\\
\ \\
Given a finite word $\und{a}=(a_{0},\cdots,a_{n})$, denote $f^{u}_{\und{a}}$ as in previous section, such that 
$$f^{u}_{\underline{a}}(z)=h^{u}(\underline{a}(h^{u})^{-1}(z)).$$
Analogously, we can describe the geometry  transverse of the unstable foliation $W^{u}(\Lambda,R)$, using a regular Cantor set $K^{s}$ define by a mapping $g^{s}$ of class $C^{r}$ (expansive of type $\Sigma_{\B}$).\\
Moreover, there exists a unique homeomorphism $h^{s}\colon \Sigma^{-}_{\B} \to K^s$ such that
$$
h^{s}(\underline{a}) \in I(a_0), \text{ for } \underline{a} = (\dots,a_1,a_0) \in \Sigma^{-}_{\mathbb{B}}
\ \ and \ \ 
h^{s}\circ\sigma^{-} =g^{s} \circ h^{s},$$
where $\sigma^{-}\colon \Sigma_{\B}^{-} \to \Sigma_{\B}^{-}$, is defined as follows $\sigma^{-}((a_{n})_{n\leq 0})=(a_{n-1})_{n\leq0}$.\\
Given a finite word $\und{a}=(a_{-n},\cdots,a_{0})$, denote $f^{s}_{\und{a}}$ as in previous section, such that 
$$f^{s}_{\underline{a}}(z)=h^{s}((h^{s})^{-1}(z)\underline{a}).$$
Also, the horseshoe $\Lambda$ is locally the product of two regular Cantor sets $K^{s}$ and $K^{u}$. So, the Hausdorff dimension of $\Lambda$, $HD(\Lambda)$ is equal to
$HD(K^{s}\times K^{u})$, but for regular Cantor sets, we have that $HD(K^{s}\times K^{u})=HD(K^{s})+HD(K^{u})$, thus $HD(\Lambda)=HD(K^{s})+HD(K^{u})$ (cf. \cite[chap 4]{PT}).
\bibliographystyle{alpha}	
\bibliography{First_Paper}

\end{document}